\documentclass[12pt,centertags, a4paper]{amsart}

\usepackage{amsmath}
\usepackage{amsthm}
\usepackage{amssymb}
\usepackage{mathrsfs} 
\usepackage[a4paper]{hyperref} 
\usepackage{graphicx}
\usepackage{exscale}
\usepackage{verbatim}
\usepackage[all,dvips,arc,curve,color,frame]{xy}
\usepackage{enumerate}
\newxyColor{pink}{1.0 0.4 0.5}{rgb}{}

\numberwithin{equation}{section}

\allowdisplaybreaks

\theoremstyle{definition}
	\newtheorem{definition}{Definition} 
	\newtheorem*{definition*}{Definition}
	\newtheorem{example}[definition]{Example}
	\numberwithin{definition}{section}
\theoremstyle{plain}
	\newtheorem{lemma}[definition]{Lemma}
	\newtheorem{proposition}[definition]{Proposition}
	\newtheorem{theorem}[definition]{Theorem}
	\newtheorem*{theorem*}{Theorem}
	\newtheorem{corollary}[definition]{Corollary}

	\newtheorem*{claim*}{Claim}
\theoremstyle{remark}

	\newtheorem{problem}{Problem}
	
	\newtheorem{remark}[definition]{Remark}

\renewcommand{\bf}{\textbf}

\renewcommand{\phi}{\varphi}

\newcommand{\C}{\mathbb{C}}
\newcommand{\Z}{\mathbb{Z}}

\newcommand{\R}{\mathbb{R}}

\newcommand{\ra}{\to}

\begin{document}
\bibliographystyle{plain}

\title{Quasi-total orders and translation numbers}
\author{Gabi Ben Simon and Tobias Hartnick}
\date{\today}
\maketitle

\begin{abstract}
We show that a group admits a non-zero homogeneous quasimorphism if and only if it admits a certain type of action on a poset. Our proof is based on a construction of quasimorphisms which generalizes Poincar\'e--Ghys' construction of the classical translation number quasimorphism. We then develop a correspondence between quasimorphisms and actions on posets, which allows us to translate properties of orders into properties of quasimorphisms and vice versa. Concerning examples we obtain new realizations of the Rademacher quasimorphism, certain Brooks type quasimorphisms, the Dehornoy floor quasimorphism as well as Guichardet-Wigner quasimorphisms on simple Hermitian Lie groups of tube type. The latter we relate to Kaneyuki causal structures on Shilov boundaries, following an idea by Clerc and Koufany. As applications we characterize those quasimorphisms which arise from circle actions, and subgroups of Hermitian Lie groups with vanishing Guichardet-Wigner quasimorphisms.
\end{abstract}

\tableofcontents

\section{Introduction}
This article continues our investigation of the relation between bi-invariant partial orders on groups and homogeneous quasimorphisms initiated in \cite{BSH}. There we observed that every homogeneous quasimorphism arises as a multiple of the growth function (in the sense of Eliashberg and Polterovich \cite{EP}) of a bi-invariant partial order, but left open the converse question, which asks for a characterization of those bi-invariant partial orders that give rise to quasimorphisms via their growth functions. In the present article we introduce a special class of bi-invariant partial orders on groups, which we call quasi-total orders. We then establish the following results:
\begin{itemize}
\item[(a)] Growth functions of quasi-total orders are homogeneous quasimorphism (see Theorem \ref{MainResult}).
\item[(b)] Conversely, every quasimorphism arises as the growth function of some quasi-total order (up to a multiplicative constant, see Theorem \ref{MainResult} and Proposition \ref{Taut}). This quasi-total order is not unique in general.
\item[(c)] Special classes of quasi-total orders give rise to special classes of quasimorphisms. 
\end{itemize}
Concerning (c) we will actually prove the following more specific statements:
 \begin{itemize}
\item[(c1)] We describe, which quasi-total orders correspond to homomorphisms (see Proposition \ref{TotalIntro}).
\item[(c2)] We describe, which quasi-total orders correspond to pullbacks of the classical rotation number via some action on the circle (see Proposition \ref{CircleIntro}); here our treatment is inspired by \cite{Ito}.  
\item[(c3)] We also discuss how topological assumptions on the order in question influence the behaviour of the corresponding quasimorphism. Here the concrete consequences are more technical to state; see Proposition \ref{TheoremHyperIntro} below.
\end{itemize}
Our approach was originally motivated from our study of quasimorphisms on (finite-dimensional) simple Lie groups. In this context we obtain notably a new interpretation of results of Clerc and Koufany \cite{ClercKoufany}, see Theorem \ref{LieSummary}. However, the methods developed in this article apply far beyond this case. For example, our construction includes the construction from \cite{Ito} as a special case and also provides new constructions of the Rademacher quasimorphism on ${\rm PSL}_2(\Z)$ and the Brooks quasimorphisms on free groups. Since quasi-total orders are induced by certain actions on posets (to be described below) we obtain from (a) and (b) above a complete characterization of groups admitting a non-zero homogeneous quasimorphism in terms of actions on posets. Unlike existing cohomological characterizations our characterization does not require any local compactness or second countability assumptions on the underlying group, and hence seems particular amenable to the study of infinite-dimensional groups.\\

We now summarize the results of this article in the same order in which they appear in the body of the text. Before we can start we have to recall some basic definitions: Given a group $G$, a partial order $\leq$ on $G$ is called \emph{bi-invariant} if $g \leq h$ implies both $gk \leq hk$ and $kg \leq kh$ for all $g,h,k \in G$. Equivalently, the associated \emph{order semigroup}
\[G^+ := \{g \in G\,|\, g \geq e\}\]
is a conjugation-invariant monoid satisfying the \emph{pointedness condition} $G^+ \cap (G^+)^{-1} = \{e\}$. In this case we call the pair $(G, \leq)$ a \emph{partially bi-ordered group} and refer to the elements of\footnote{Here and in the sequel we distinguish the set $\mathbb N := \{1,2, \dots\}$ of positive integers and the set of non-negative integers $\mathbb N_0 := \mathbb N \cup \{0\}$.}  
\[G^{++} := \{g \in G^+ \setminus\{e\}\,|\, \forall h \in G \exists n \in \mathbb N_0:\, g^n \geq h\}\]
as \emph{dominants} of the ordered group $(G, \leq)$. We will always assume that $(G, \leq)$ is \emph{admissible} meaning that $G^{++} \neq \emptyset$. In this case the \emph{relative growth function} 
\[\gamma: G^{++} \times G \to \R, \quad (g,h) \mapsto \gamma(g,h)\]
given by 
\[\gamma(g,h) := \lim_{n \to \infty}\frac{\inf\{p \in \Z\,|\, g^p \geq h^n\}}n\]
contains valuable numerical information on the order. (This function was introduced in \cite{EP}.) Fixing $g \in G^{++}$ we obtain a \emph{growth function} $\gamma_g: G \to \R$, $\gamma_g(h) := \gamma(g,h)$. This function is always \emph{homogeneous}, i.e. satisfies $\gamma_g(h^n) = n \gamma_g(h)$ for all $n \in \mathbb N$; here we ask whether it happens to be a quasimorphism, i.e. whether it satisfies
\[D(\gamma_g) := \sup_{h,k \in G}|\gamma_g(hk) - \gamma_g(h) -\gamma_g(k)| < \infty.\]
If this is the case, we refer to the constant $D(\gamma_g)$ as the \emph{defect} of $\gamma_g$. With this terminology understood we can now state the problem to be solved in this article:
\begin{problem}\label{MainProblem} Describe a class $\mathcal C$ of bi-invariants partial orders on groups such that
\begin{itemize}
\item[(i)] the growth functions of any order $\leq \in \mathcal C$ are nonzero homogeneous quasimorphisms;
\item[(ii)] every nonzero homogeneous quasimorphism arises as the growth function of some $\leq \in \mathcal C$ (say, up to a positive multiple).
\end{itemize}
\end{problem}
We now aim to describe a class of orders which solves Problem \ref{MainProblem}. We start by observing that bi-invariant orders on $G$ arise from (effective) $G$-actions on posets (not necessarily order-preserving). Indeed, if $G$ acts effectively on a poset $(X, \preceq)$ then we obtain a bi-invariant partial order $\leq$ on $G$ by setting
\[g \leq h :\Leftrightarrow \forall k \in G\; \forall x\in X: \, (kg).x \preceq (kh).x.\]
We refer to $\leq$ as the order \emph{induced} from the $G$-action on $(X, \preceq)$. A priori this order may be trivial, but it will be non-trivial in the cases we discuss below. Since every bi-invariant partial  order is induced from itself (via the left action of $G$ on itself), specifying a class of bi-invariant partial orders is equivalent to specifying a class of effective $G$-posets.
\begin{definition}\label{DefHalfspace}
Let $X$ be a set. A family of subsets $\{H_n\}_{n \in \mathbb Z}$ of $X$ is called a \emph{half-space filtration} of $X$ if
\begin{itemize}
\item[(H1)] $H_{n+1} \subsetneq H_{n}$, $n \in \mathbb Z$.
\item[(H2)] $\bigcap H_{n} = \emptyset$, $\bigcup H_n = X$. 
\end{itemize}
Given a half-space filtration $\{H_n\}_{n \in \mathbb Z}$ of $X$ and elements $a, b \in X$ we denote by $h(a) := \sup \{n \in \Z\,|\, a \in H_n\}$ the \emph{height} of $a$ and by $h(a,b) := h(a) - h(b)$ the \emph{relative height} of $a$ over $b$ with respect to $\{H_n\}$. A triple $(X, \preceq, \{H_n\}_{n \in \mathbb Z})$ is called a \emph{half-space order} of \emph{width} bounded by $w := w(X, \preceq, \{H_n\})$ if $(X, \preceq)$ is a poset, $\{H_n\}$ is a half-space filtration of $X$ and
\[\forall a,b \in X: h(a,b) \geq w \Rightarrow a \succeq b.\]
In this case we say that a group $G$ acts on $(X, \preceq, \{H_n\}_{n \in \mathbb Z})$ by \emph{quasi-automorphisms} of \emph{defect} bounded by $d$ if 
\[\forall g \in G, a,b \in X:\,|h(ga,gb)-h(a,b)| \leq d.\]
If $G$ moreover acts effectively, then the induced order $\leq$ on $G$ is called a \emph{quasi-total order} of defect bounded by $d$ provided the action is \emph{unbounded} in the sense that
\begin{eqnarray}\label{Unbounded action}
\exists g \in G, a \in X:\,\lim_{n \to \pm \infty} h(g^n.a) = \pm \infty.\end{eqnarray}
\end{definition}
Note that the $G$-action on $X$ does not necessarily preserve the order; however the above assumptions guarantee that the order is \emph{quasi-preserved} in some sense. Using the above terminology we can now state our solution to Problem \ref{MainProblem}:
 \begin{theorem}\label{MainResult}
The class $\mathcal C$ of quasi-total orders solves Problem \ref{MainProblem}, i.e.
\begin{itemize}
\item[(i)] the growth functions of any quasi-total order $\leq$ are nonzero homogeneous quasimorphisms;
\item[(ii)] every nonzero homogeneous quasimorphism on a given group $G$ arises as the growth function of some quasi-total order on that group (up to a positive multiple).
\end{itemize}
 \end{theorem}
 In this article we will describe the following subclasses of the class of quasi-total orders more closely (definitions will be given below; other classes of examples will be discussed in subsequent work):
\begin{itemize}
 \item the standard halfspace order on the real line, corresponding to the classical translation number; 
 \item quasi-total orders induced from planar group embeddings; examples include the Rademacher quasimorphism and various Brooks type quasimorphisms;
 \item quasi-total orders induced from quasi-total triples; this largest class includes the following subclasses:
 \begin{itemize}
 \item bi-invariant admissible total orders (which lead to homomorphisms as growth functions);
 \item orders induced from total triples of the form $(G, \preceq, T)$; these are closely related to actions on the circle;
 \item orders induced from smooth quasi-total orders. These are related to causal structures on manifolds; examples include Guichardet-Wigner quasimorphisms on Hermitian Lie groups of tube type.
 \end{itemize}
 \end{itemize}
We now turn to each of this classes individually: The simplest example of a half-space order is given by $(\R, \leq, \{[n, \infty)\})$, where $\leq$ denotes the usual total order on $\R$. If ${\rm Homeo}^+(S^1)$ denotes the group of orientation-preserving homeomorphisms of the circle and ${\rm Homeo}^+_{\Z}(\R)$ its universal covering group, then the ${\rm Homeo}^+(S^1)$-action on the circle lifts to an action of ${\rm Homeo}^+_{\Z}(\R)$ on the real line. This action turns out to be by quasi-automorphisms with respect to $(\R, \leq, \{[n, \infty)\})$, hence gives rise to a quasi-total order, and thereby to a quasimorphism $T$ on ${\rm Homeo}^+_{\Z}(\R)$. This quasimorphism can be traced all the way back to the work of Poincar\'e \cite{Po1, Po2}, where a $\R/\Z$-valued continuous map on ${\rm Homeo}^+(S^1)$ is constructed. Namely, it turns out that this famous \emph{rotation number} lifts to a real-valued map on the universal covering group  ${\rm Homeo}^+_{\Z}(\R)$ of ${\rm Homeo}^+(S^1)$. This lift, which is sometimes called \emph{translation number}, is precisely our quasimorphism $T$. The observation that the translation number is a quasimorphism appears first in Ghys' fundamental work on group actions on the circle \cite{Ghys1, Ghys2, GhysEnglish, BargeGhys}, where it is  related to the universal bounded Euler class. Our proof of Part (i) of Theorem \ref{MainResult} is based on a far-reaching abstraction of these classical arguments. In particular, we will construct for every half-space order a corresponding \emph{generalized translation number} (see Subsection \ref{SecTrans}), which is then shown to coincide (up to a multiple) with the growth functions of the corresponding quasi-total order (see Subsection \ref{SecGrowth}).\\

\begin{figure}
	\centering
  \includegraphics{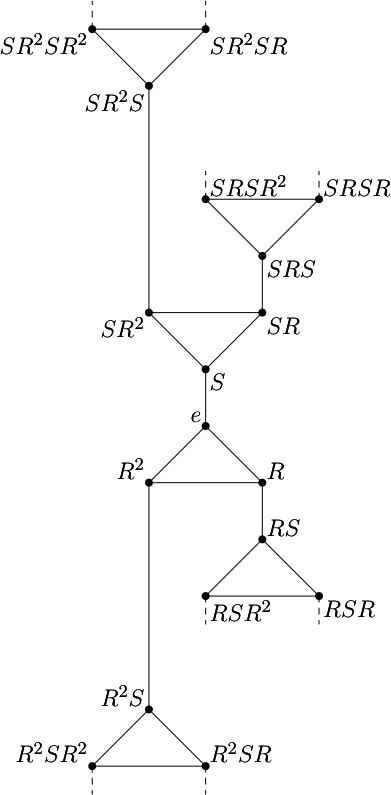}
	\caption{A planar embedding of ${\rm PSL_2(\Z)}$ corresponding to the Rademacher quasimorphism}
	\label{FigureRademacher}
\end{figure}

While $(\R, \leq, \{[n, \infty)\})$ is the most classical examples of a half-space order, it is special in many respects, for example because of the totality of $\leq$. A more typical example of a half-space order is obtained from the real plane $\R^2$ by defining an order $\preceq$ by
 \[(x, y) \prec (x', y') :\Leftrightarrow x < x'\]
 and halfspaces $H_n$ by
 \[H_n := \{(x, y) \in \R^2\,|\, x \geq n\}.\]
More generally, if $X \subset \R^2$ is any subset, which intersects all $H_n \setminus H_{n+1}$ non-trivially, then $(X, \preceq|_{X \times X}, H_n \cap X)$ is half-space ordered. For example, the group ${\rm PSL}_2(\Z) = \Z/2\Z \ast \Z/3\Z$ can be embedded into the plane by continuing the pattern given in Figure \ref{FigureRademacher}. It is easy to see that the left-action of ${\rm PSL}_2(\Z)$ on its planar image is by quasi-automorphisms with respect to the resulting half-space order; we thus obtain a quasimorphism on ${\rm PSL}_2(\Z)$ just from drawing the picture! This quasimorphism happens to be the famous \emph{Rademacher quasimorphism}; see the discussion in Example \ref{Rademacher} below.\\

A different class of examples arises from bi-invariant total orders; however, these do not provide interesting quasimorphisms (see Subsection \ref{SecTotal}):
 \begin{proposition}\label{TotalIntro}
 Every admissible bi-invariant total order on a group $G$ is quasi-total; its associated growth functions are homomorphisms.
 \end{proposition}
A more interesting class of examples can be obtained by weakening the condition of totality just a little bit:
\begin{definition}
Let $(X, \preceq)$ be a poset and denote by ${\rm Aut}(X, \preceq)$ the group of order-preserving permutations of $X$. Then an element $T \in {\rm Aut}(X, \preceq)$ is called \emph{dominant} if
  \begin{eqnarray*}
 \forall a,b \in X \; \exists n \in \mathbb N:\; T^{n}a \succ b.
 \end{eqnarray*} 
Given a poset $(X, \preceq)$ and a dominant automorphism $T$ we call the  triple $(X, \preceq, T)$ a \emph{quasi-total triple} if
 \begin{eqnarray*}\label{QuasiTotal}
 \exists N(X) \in \mathbb N\; \forall a,b \in X \; \exists 0 \preceq k \preceq N(X): \; (a \preceq T^kb \;\vee\; b \preceq T^k a). 
 \end{eqnarray*}
 The quasi-total triple $(X, \preceq, T)$ is called \emph{complete} if $\forall a \in X: a \preceq Ta$.
 An order-preserving action of $G$ on $(X, \preceq)$ is called \emph{dominating} if it commutes with $T$ and satisfies
 \begin{eqnarray}\label{AA}
 \exists gÊ\in G\; \exists x \in X\; \exists n \in \mathbb N: g.x \succeq T^n.x.
 \end{eqnarray}
\end{definition}
The link between quasi-total triples and quasi-total orders will be established in Subsection \ref{SecTriples} below.
\begin{proposition}
Let $(X, \preceq, T)$ be a complete quasi-total triple, $x_0 \in X$ a basepoint and 
\[H_n := \{x \in X\,|\, x \succeq T^n.x_0\}.\]
Then $(X, \preceq, \{H_n\})$ is a half-space order. Moreover, if $G$ acts dominatingly and effectively by automorphisms on $(X, \preceq, T)$, then it acts unboundedly and by quasi-automorphisms on $(X, \preceq, \{H_n\})$. In particular, the order induced by $\preceq$ on $G$ is quasi-total.
\end{proposition}
\begin{definition}
A quasi-total order $\leq$ is called \emph{special} if it is induced from an effective, dominating action of $G$ on a quasi-total triple. This quasi-total triple is then said to \emph{realize} $\leq$ (or its growths functions).
\end{definition}
We will establish the following strengthening of the existence part of Theorem \ref{MainResult} in Subsection \ref{SecTaut} below:
\begin{proposition}\label{TautReal}\label{Taut}
Every nonzero quasimorphisms arises as the multiple of a growth function of a special quasi-total order. In fact, it arises  as the multiple of a growth functions of an order induced from a quasi-total triple of the form $(G, \preceq, T)$ via the left action of $G$ on itself.
\end{proposition}
Proposition \ref{TautReal} will be established by constructing for every given quasimorphism $f$ an explicit quasi-total triple called the \emph{tautological realization}, which realizes $f$. This will lead us in Subsection \ref{SecIncomp} to the following characterization result:
\begin{corollary} A group $G$ admits a non-zero homogeneous quasimorphism if and only if if acts dominatingly on some quasi-total triple $(X, \preceq, T)$.
\end{corollary}
Here we do not assume that the action is effective, nor that the triple is complete. We have seen above that we may restrict attention to orders induced from quasi-total triples of the form $(G, \preceq, T)$. It is not possible in general to refine $\preceq$ into a total left order on $G$. Indeed, the corresponding class of quasimorphisms is rather special:\begin{proposition}\label{CircleIntro}
Assume that the order $\leq$ on $G$ is induced from a quasi-total triple $(G, \preceq, T)$ with $(G, \preceq)$ total. Then there exists a homomorphism $\phi: G \to {\rm Homeo}^+_{\Z}(\R)$ such that
 the growth functions of $\leq$ are proportional to the pullback of the translation number via $\phi$.
\end{proposition}
See Section \ref{SecCircle} for a circle of related ideas. (Note that the proposition concerns bi-invariant \emph{partial} orders on $G$ \emph{induced} from total left-orders on itself; these are not to be confused with bi-invariant total orders on $G$, which we discussed in Proposition \ref{TotalIntro} above.)\\

While we have seen that every quasimorphism can be realized using its tautological realization, there are many reasons to look for \emph{non-tautological} realizations of known quasimorphisms. Firstly, these appear often more naturally; more importantly, additional properties of a realization can often be used to establish corresponding properties of the quasimorphism. In order to illustrate this principle, we discuss the case of globally hyperbolic quasi-total triples. Here a quasi-total triple $(X, \leq, T)$ is called \emph{globally hyperbolic} if $X$ comes equipped with a topology such that the finite order intervals
\begin{eqnarray*}
&&{}[x,y] := \{z \in X\,|\, x \leq z \leq y\}
\end{eqnarray*}
are compact and the infinite order intervals
\begin{eqnarray*}
&&{}[x,\infty) := \{z \in X\,|\, x \leq z\}, \quad (-\infty, x] := \{z \in X\,|\,z \leq x\}.
\end{eqnarray*}
are closed. Let us call a subset $B \subset X$ \emph{bounded} if its closure is compact. Then we have (see Theorem \ref{TheoremHyper}): 
\begin{proposition}\label{TheoremHyperIntro}
Let $(X, \leq, \{H_n\})$ be a halfspace order such that $(X, \leq)$ is globally hyperbolic. Let $H$ be a group acting effectively and unboundedly on $(X, \leq, \{H_n\})$ and denote by $T_X:H \to \R$ the associated translation number. Then the following are equivalent for a subgroup $G < H$:
\begin{itemize}
\item[(i)] $(T_X)|_G \equiv 0$.
\item[(ii)] Every $G$-orbit in $X$ is bounded.
\item[(iii)] There exists a bounded $G$-orbit in $X$.
\end{itemize}
\end{proposition}
The proposition motivates the question whether a given quasimorphism arises as the growth function associated with a globally hyperbolic quasi-total triple. We study this question in the context of quasimorphisms on simple Lie groups. More precisely, let $\mathcal D$ be an irreducible bounded symmetric domain with Shilov boundary $\check S$ and let $G$ be the identity component of the isometry group of $\mathcal D$ with respect to the Bergman metric. Then any infinite covering of $G$ admits a (unique up to multiples) homogeneous quasimorphism, called Guichardet-Wigner quasimorphism \cite{GuichardetWigner, surface}. In general, we do not know whether this quasimorphism can be realized using a globally hyperbolic quasi-total triple. However, if the domain $\mathcal D$ happens to be of tube type, then by a result of Kaneyuki \cite{Kaneyuki} there is a unique (up to inversion) $G$-invariant causal structure on $\check S$, which gives rise to a partial order $\preceq$ on the universal covering $\check R$ of $\check S$. The fact that this order is closely related to the Guichardet-Wigner quasimorphism was first observed by Clerc and Koufany \cite{ClercKoufany}. In the language of the present paper their results can be reinterpreted as follows: Denote by $T$ the unique deck transformation of the covering $\check R \to \check S$ which is non-decreasing with respect to $\preceq$. Also denote by $\check G$ the unique cyclic covering of $G$ which acts transitively and effectively on $\check R$. Then we have:
\begin{theorem}\label{LieSummary}
The triple $(\check R, \preceq, T)$ is a globally hyperbolic quasi-total triple, which induces a quasi-total order $\leq$ on $\check G$. The growth functions of $\leq$ are proportional to the Guichardet-Wigner quasimorphism on $\check G$.
\end{theorem}
Theorem \ref{LieSummary} will be proved in Corollary \ref{GuiWi} below. We remark that the order $\preceq$ used in the present article is the \emph{closure} of the order considered in \cite{ClercKoufany}. For this \emph{closed} order we establish global hyperbolicity in Corollary \ref{KaneyukiHyperbolic}. The combination of Theorem \ref{LieSummary} and Proposition \ref{TheoremHyperIntro} yields the following result:
\begin{corollary}
Let $H < \check G$ be a subgroup. Then the Guichardet-Wigner quasimorphism vanishes along $H$ if and only if some (hence any) $H$-orbit in $\check R$ is bounded.
\end{corollary}

\bf {Acknowledgement:}  We would like to thank Uri Bader, Marc Burger, Danny Calegari, \'Etienne Ghys, Pascal Rolli and Mark Sapir for comments and remarks related to the present article. We also thank Andreas Leiser and Tobias Struble for help with the pictures.The authors also acknowledge the hospitality of IHES, Bures-sur-Yvette, and Hausdorff Institute, Bonn. The first named author is grateful to Departement Mathematik of ETH Z\"urich  and in particular Dietmar Salamon and Paul Biran for the support during this academic year. The second named author was supported by SNF grants PP002-102765 and 200021-127016.

\section{Foundations of quasi-total orders}

\subsection{The translation number of a quasi-total order}\label{SecTrans}

Throughout this subsection let $(X, \preceq, \{H_n\})$ denote a fixed half-space order of width $w$ and let $G$ be a group acting by quasi-automorphisms of defect $d$ on $(X, \preceq, \{H_n\})$ (see Definition \ref{DefHalfspace}).
\begin{proposition}
The functions $\{f_a: G \to X\}_{a \in X}$ given by $f_a(g) := h(ga, a)$ are mutually equivalent quasimorphisms of defect $\leq d$. In fact, their mutual distances are uniformly bounded by $d$.
\end{proposition}
\begin{proof} We first show that the $f_a$ are mutually at bounded distance:
\begin{eqnarray*}
|f_a(g) -f_b(g)| &=& |h(ga)-h(a)-h(gb)+h(b)|\\ &=& |h(ga)-h(gb)-(h(a)-h(b))|\\
 &=& |h(ga, gb)-h(a,b)| < d.
\end{eqnarray*}
Now let us use this fact to show that they are quasimorphisms:
\begin{eqnarray*}
&&|f_a(gk)-f_a(g)-f_a(k)|\\ &\leq& |f_a(gk) -f_{ka}(g) - f_a(k)| + d\\
&=&  |h(gka)-h(a) -h(gka)+h(ka) -h(ka)+h(a) | + d\\
&=& d.
\end{eqnarray*}
\end{proof}
The assumption that the action of $G$ on $X$ is unbounded implies immediately that each of the functions $f_a$ is unbounded. Then standard properties of homogeneization (see e.g. \cite{scl}) yield the following results
\begin{corollary}
There exists a nonzero homogeneous quasimorphism $T_{(X, \preceq, \{H_n\})}:G \to \R$ of defect $\leq 2d$ such that for all $a \in X$,
\begin{eqnarray*}
T_{(X, \preceq, \{H_n\})}(g) = \lim_{n \to \infty} \frac{h(g^na, a)}{n}.
\end{eqnarray*}
\end{corollary}
\begin{definition}
The quasimorphism $T_{(X, \preceq, \{H_n\})}: G \to \R$ is called the \emph{translation number} associated with the action of $G$ on $(X, \preceq, \{H_n\})$.
\end{definition}

\subsection{Translation numbers of growth functions}\label{SecGrowth}
Throughout this subsection let $(X, \preceq, \{H_n\})$ denote a fixed half-space order of width $w$ and let $G$ be a group acting by quasi-automorphisms of defect $d$ on $(X, \preceq, \{H_n\})$. We denote by $T_X := T_{(X, \preceq, \{H_n\})}$ the translation number associated with this action and by $\leq$ the induced order on $G$. Our goal is to establish the following result:
\begin{theorem}\label{GrowthMain}
The growth functions of $\leq$ are multiples of $T_X$, in particular, they are nonzero homogeneous quasimorphisms.
\end{theorem}
For the proof we need to recall some basic results and concepts from \cite{BSH}. Given a partially bi-ordered group $(G, \leq)$ with order semigroup $G^+$ and a homogeneous quasimorphism $f: G \to \R$, we say that $f$ \emph{sandwiches} $G^+$ if for some $C >0$
\begin{eqnarray}\label{sandwich}
\{g \in G\,|\, f(g) \geq C\} \subset G^+.\end{eqnarray}
(In this case automatically $G^+ \subset \{g \in G\,|\, f(g) \geq 0\}$, hence the name.) Then we have the following result (\cite[Prop. 3.3]{BSH}):
\begin{proposition}\label{BSH1}
Suppose that $(G, \leq)$ is a partially bi-ordered group and that $f: G \ra \R$ is a non-zero homogeneous quasi-morphism. If $f$ sandwiches $\leq$, then $\leq$ is admissible and
for all $g \in G^{++}$, $h \in G$ we have
\[\gamma(g,h) = \frac{f(h)}{f(g)}.\]
\end{proposition}
Thus Theorem \ref{GrowthMain} is reduced to establishing the following proposition:
\begin{proposition}\label{CompletionPseudoComplete} With notation as above, the quasimorphism $T_X$ sandwiches the partial order $\leq$.
\end{proposition}
\begin{proof} The quasimorphisms $f_a$ are mutually at uniformly bounded distance $d$, hence at distance $d$ from $T_{X,\{H_n\}}$. Now assume $T_{X}(g) > w+2d$. Then for all $k \in G$, $x \in X$ we have
\[h(kg.x, k.x) \geq h(g.x, x) -d \geq f_x(g)-d \geq T_{X}(g) -2d > w, \]
hence $kg.x \succeq k.x$ by definition of $w$. This in turn implies $g \geq e$.
\end{proof}
This finished the proof of Theorem \ref{GrowthMain} and thereby establishes Part (i) of Theorem \ref{MainResult}.

\subsection{Global hyperbolicity} It is well-known that the classical translation number contains valuable information about orbits of subgroups of ${\rm Homeo}^+_\Z(\R)$ on the real line and circle. For example, the question whether a subgroup $H < {\rm Homeo}^+_\Z(\R)$ has a bounded orbit in $\R$ can be decided from the translation number. Indeed, such a bounded orbit exists if and only if $T_\R|_H \equiv 0$;
in this case, in fact all orbits are bounded. In order to obtain similar results for other types of quasimorphisms additional topological assumptions are necessary; concerning the existence of bounded orbits, global hyperbolicity is the key property. Indeed, we have the following result, which was stated as Proposition \ref{TheoremHyperIntro} from the introduction:
\begin{theorem}\label{TheoremHyper}
Let $(X, \leq, \{H_n\})$ be a halfspace order such that $(X, \leq)$ is globally hyperbolic. Let $H$ be a group acting effectively and unboundedly on $(X, \leq, \{H_n\})$ and denote by $T_X:H \to \R$ the associated translation number. Then the following are equivalent for a subgroup $G < H$:
\begin{itemize}
\item[(i)] $(T_X)|_G \equiv 0$.
\item[(ii)] Every $G$-orbit in $X$ is bounded.
\item[(iii)] There exists a bounded $G$-orbit in $X$.
\end{itemize}
\end{theorem}
Before we turn to the proof we observe that global hyperbolicity can be characterized in terms of halfspaces:
\begin{lemma}
Let $(X, \leq, \{H_n\})$ be a halfspace order; then $(X, \preceq)$ is hyperbolic if and only if order intervals are closed and the sets $\overline{H_n \setminus H_{n+1}}$ are compact.
\end{lemma}
\begin{proof} Denote by $w$ the width of $(X, \leq, \{H_n\})$. Given $n \in \mathbb N$ let $x_n^+ \in H_{n+w+1}$ and $x_n^- \in H_{n-w-1} \setminus H_{n-w}$. Then 
\[H_n \setminus H_{n+1} \subset [x_n^-, x_n^+],\]
hence global hyperbolicity implies compactness of the sets $\overline{H_n \setminus H_{n+1}}$. For the converse observe that if $x \in H_n$, $y \in H_m$, then
\[[x,y] \subset \bigcup_{k= n-w-1}^{m+w+1} \overline{H_k \setminus H_{k+1}}\]
\end{proof}
\begin{proof}[Proof of Theorem \ref{TheoremHyper}] (i)$\Rightarrow$(iii): We claim that (i) implies that there are no $g \in G, n \in \mathbb N, x \in X$ satisfying
\begin{eqnarray*}
h(g^nx, x) \geq 2d.
\end{eqnarray*}
Indeed, otherwise, we had for all $m \in \mathbb N$ the inequaliy
\begin{eqnarray*}
h(g^{nm}.x, g^{n(m-1)}.x) \geq h(g^nx, x)-d \geq d,
\end{eqnarray*}
whence inductively
\begin{eqnarray*}
h(g^{nm}.x, x) \geq h(g^{nm}.x, g^{n(m-1)x}) + h(g^{n(m-1)}.x, x) \geq md,
\end{eqnarray*}
which leads to
\begin{eqnarray*}
\frac{h(g^{nm}.x, x) }{nm} \geq \frac d n,
\end{eqnarray*}
and thus $T_X(g) \geq \frac d n$ by passing to the limit $m \to \infty$. This contradiction shows that
\[
\forall g \in G, n \in \mathbb N, x \in X: \, h(g^nx, x) \leq 2d.
\]
Applying the same argument to the reverse order, we can strengthen this to 
\[
\forall g \in G, n \in \mathbb N, x \in X: \, |h(g^nx, x)| \leq 2d.
\]
This implies that each orbit is contained in a finite number of strips of the form $H_n \setminus H_{n+1}$, hence bounded by the lemma.
(ii)$\Rightarrow$(iii): obvious.\\
(iii)$\Rightarrow$(i): Assume that $\overline{G.x}$ is compact and let $g \in G$. Consider the sequence $x_n := g^{n}.x$.  We claim that there exist $n_-, n_+$ (possibly depending on $g$ and $x$) such that
\begin{eqnarray*}
\{x_n\} \subset H_{n_+} \setminus H_{n_-}.
\end{eqnarray*}
Observe first that the claim implies that $h(g^nx, x)$ is bounded, whence $T_X(g) = 0$; it thus remains to establish the claim. Assume that the claim fails; replacing the order by its reverse if necessary we may assume that $h(g^nx, x)$ is not bounded above. We thus find a subsequence $n_k$ such that for every $y \in X$ there exists $k(y)$ such that for all $k > k(y)$ we have $g^{n_k}x \succeq y$. By compactness of $\overline{G.x}$ there exists an accumulation point $x_\infty$ of $x_n$ and since order intervals are closed we have $x_\infty  \geq y$ for all $y \in X$. However, a halfspace order does not admit a maximum.
\end{proof}

\section{Constructions of quasi-total orders}

\subsection{Complete quasi-total triples}\label{SecTriples} Throughout this subsection, let $(X, \preceq, T)$ be a complete quasi-total triple. We choose a basepoint $x_0 \in X$ and define halfspaces $H_n$ for $n \in \mathbb Z$ by
\[H_n := [T^n.x_0, \infty).\]
Since $T$ is dominating, these form indeed a half-space filtration for $X$. Let us describe the height function of $(X, \{H_n\})$ in terms of $T$ and $x_0$: Since the order is complete, there exists an absolute constant $C(X)$ such that for all $a,b \in X$ we have
\[(a \preceq T^{C(X)}b) \vee (b \preceq T^{C(X)}a).\]
Assume $h(a) = n$. Then $a \succeq T^n.x_0$ and $a \not \succeq T^{n+k}.x_0$ for $k > 0$, whence $a \preceq T^{n+(2C(X)+1)}.x_0$. Thus we have established:
\begin{lemma}
If $h(a) = n$, then
\[T^n.x_0 \preceq a \preceq T^{n+(2C(X)+1)}.x_0.\]
\end{lemma}
This property determines $h$ up to a bounded error. Moreover, we can deduce:
\begin{lemma}
The triple $(X, \preceq, \{H_n\})$ is a half-space order.
\end{lemma}
\begin{proof} Assume
\[h(a,b) > 2C(X)+1,\]
and let $n:= h(a)$, $m := h(b)$. By the lemma we then have
\[a \succeq T^n.x_0 \succeq T^{m+(2C(X)+1)}.x_0 \succeq b.\]
\end{proof}
Now assume $G$ acts by automorphisms on the triple $(X, \preceq, T)$. Then we have:
\begin{proposition}
For all $g \in G$, 
$|h(ga, gb) - h(a,b)| < 4 C(X)+2$.
\end{proposition}
\begin{proof} Assume $h(a) = n$ and $h(b) = m$. Then we have $h(a,b) = n-m$, and
\[T^n.x_0 \preceq a \preceq T^{n+(2C(X)+1)}.x_0, \quad  T^m.x_0 \preceq b \preceq T^{m+(2C(X)+1)}.x_0,\]
and consequently
\[T^n.gx_0 \preceq ga \preceq T^{n+(2C(X)+1)}.gx_0, \quad  T^m.gx_0 \preceq gb \preceq T^{m+(2C(X)+1)}.gx_0.\]
Now we find $k \in \mathbb Z$ such that
\[T^k x_0 \preceq g.x_0 \preceq T^{k+(2C(X)+1)}.x_0;
\]
inserting this into the previous set of inequalities we obtain
\[T^{n+k}.x_0 \preceq ga \preceq T^{n+k+(4C(X)+2)}.x_0, \quad  T^{m+k}.x_0 \preceq gb \preceq T^{m+(4C(X)+2)}.x_0.\]
We deduce  that
\[n+k \leq h(ga) \leq n+k+(4C(X)+2), \quad m+k \leq h(gb) \leq m+k+(4C(X)+2),\]
and hence
\[(n-m)-(4C(X)+2) \leq h(ga)-h(gb) \leq (n-m)+(4C(X)+2),\]
which is to say
\begin{eqnarray*}
|h(ga, gb)-h(a,b)| \leq 4C(X)+2.
\end{eqnarray*}
\end{proof}
\begin{corollary}
Assume $G$ acts by automorphisms on the quasi-total triple $(X, \preceq, T)$. Then $G$ acts by quasi-automorphisms on the associated half-space order $(X, \preceq, \{H_n\})$.
\end{corollary}
We will denote the translation number of the triple $(X, \preceq, \{H_n\})$ by
\[
T_{(X, \preceq, T)} := T_ {(X, \preceq, \{H_n\})}.
\]
Let us show that this translation number does not depend on the choice of basepoint $x_0$: For this we define the relative $T$-height of $a,b \in X$ by the formula \[h_T: X^2 \to \Z, \quad h_T(a,b) := \inf\{m \in \Z\,|\, T^mb \geq a\}.\]
Then unravelling definitions yields $|h_T(a,b)-h(a,b)| < 2C(X)+2$. Thus we obtain:
\begin{proposition}
The translation number $T_{(X, \preceq, T)}$  is given by the formula
\begin{eqnarray*}
T_{(X, \preceq,T)}(g) = \lim_{n \to \infty} \frac{h_T(g^na, a)}{n};
\end{eqnarray*}
in particular, it is independent of the choice of basepoint $x_0$.
\end{proposition}

\subsection{Tautological realization}\label{SecTaut}
We  now turn to a proof of Proposition \ref{Taut}, whose statement we repeat here for ease of reference:
\begin{proposition}\label{Taut2}
Every nonzero quasimorphisms arises as the multiple of a growth functions of an order induced from a quasi-total triple of the from $(G, \preceq, T)$ via the left action of $G$ on itself.
\end{proposition}
\begin{proof}
Let $G$ be a group and $f$ a nonzero (hence unbounded) homogeneous quasimorphism on $G$. Set
\begin{eqnarray}
x \prec_f y &:\Leftrightarrow& \forall g \in G:\; f(gx) < f(gy),
\end{eqnarray}
and let $h \in G$ be an element with $f(h) > $. Denote by $\rho_h$ the right-multiplication by $h$. Then $(G, \preceq_f, \rho_h)$ is a quasi-total triple. In view of Proposition \ref{BSH1} it remains to show that $f$ sandwiches the partial order $\leq_f$ induced by $(G, \preceq_f, \rho_h)$ on $G$. For this let $g \in G$ with $f(g) > D(f)$; then for all $x \in G$ we have
\[\forall h \in G: f(ghx) \geq f(g) +f(hx) -D(f) > f(hx),\]
hence $gx \succeq_f x$ and thus $g \geq_f e$, finishing the proof \end{proof}
Note that in view of the remarks following Theorem \ref{GrowthMain} we have now completed the proof of Theorem \ref{MainResult}.

\subsection{Elementary constructions}\label{SecElem}
The notion of a quasi-total triple is closed under various elementary constructions. The following three persistence properties are immediate from the definition:
\begin{proposition}
\begin{itemize}
\item[(i)] (Lexicographic products) Let $(X_0, \preceq_0, T_0)$ be a complete quasi-total triples and $(X_i, \preceq_i)_{i \in \mathbb N}$ be a family of arbitrary posets. On $X := \prod_{i=0}^\infty X_i$ define the lexicographic ordering by
\[(x_i) \prec (y_i) : \Leftrightarrow \exists j \in \mathbb N_0: (\forall i < j: x_i = y_i) \wedge (Œx_{j} \prec_j y_j)\]
and define $T: X \to X$ by $T(x_i) := (T_0x_0, x_1, x_2, \dots)$. Then $(X, \preceq, T)$ is a quasi-total triple.
\item[(ii)] (Subtriples) Let $(X, \preceq, T)$ be a complete quasi-total triple. Let  $Y \subset X$ be a subset and suppose there exists $S \in {\rm Aut}(Y, \preceq)$ such that for all $y \in Y$ we have $Sy \succeq Ty$. Then $(Y, \preceq, S)$ is quasi-total.
\item[(iii)] (Refinement) Let $(X, \preceq, T)$ be a quasi-total triple and $\preceq'$ be a refinement of $\preceq$; then $(X, \preceq', T)$ is a quasi-total triple.
\end{itemize}
\end{proposition}
These constructions, trivial as they may seem, immediately give rise to a large supply of interesting quasimorphisms. We illustrate this in the following example:
\begin{example}\label{Twisting}
Let $(X, \preceq, T) = (\R, \leq, x \mapsto x+1)$ be the standard quasi-total triple, which realizes the classical translation number  $T_\R$, i.e. the lift of Poincar\'e's rotation number. Then for any set $X$ (which we consider as a trivial poset $(X, =)$) we obtain a new quasi-total triple $(C_X := \R \times X, \preceq, T)$ as the lexicographic product. Explicitly, we have
\[(\lambda, x) \preceq (\mu, y) : \Leftrightarrow \lambda < \mu\]
and $T(\lambda, x) := (\lambda+1, x)$. This quasi-total triple induces a quasimorphism $T_{C_X} =T_{(C_X, \preceq, T)}$ on $G := {\rm Aut}(C_X, \preceq, T)$, hence on any subgroup of $G$. The reader might have the impression that these quasimorphisms are a trivial variation of $T_\R$, but this is not the case. For concreteness, let $X := S^1$; then it is easy to see that $T_{C_X}$ cannot be the pullback of $T_\R$ via any embedding $G \to {\rm Homeo}_\Z^+(\R)$. Indeed, such an embedding does not exist, since ${\rm Homeo}_\Z^+(\R)$ is torsion-free.  
\end{example}
The notion of a half-space order is even more flexible. The following trivial observation is important:
\begin{lemma}
Let $(X, \preceq, \{H_n\})$ be a halfspace order and $Y \subset X$ a subset such that $Y \cap (H_n \setminus H_{n+1}) \neq \emptyset$ for all $n$. Then $(Y, \preceq|_{Y\times Y}, \{H_n \cap Y\})$ is  a halfspace order. 
\end{lemma}
In the next section we will apply the following special case:
\begin{corollary}
Let $Y$ be a set. Then an embedding $\iota:Y \hookrightarrow \R^2$ induces a halfspace order on $Y$ by setting
\[y_1 \prec y_2 :\Leftrightarrow \iota(y_1) = (x_1, z_1), \iota(y_2) = (x_2, z_2), x_1 < x_2\]
and 
\[H_n := \{y \in Y\,|\, \iota(y) = (x,z), x \geq n\},\]
provided $H_n \setminus H_{n+1} \neq \emptyset$ for all $n \in \mathbb Z$.
\end{corollary}

\subsection{Planar embeddings} In view of the last corollary, a good strategy to construct quasimorphisms on groups is making the group act on a subset of the plane. One way to do so is to embed the group itself into the plane. If the embedding $\iota: G \to \R^2$ is chosen in such a way that the left action of $G$ on itself induces an unbounded action on $\iota(G)$ by quasi-automorphisms, then we obtain a quasi-total order, hence a quasimorphism. We provide two examples where this strategy works:
\begin{example}[Rademacher quasimorphism]\label{Rademacher} Let $G = {\rm PSL_2(\Z)} = \Z/2\Z \ast \Z/3\Z$. Denote by $S$ a generator of $\Z/2\Z$ and by $R$ a generator of $\Z/3\Z$, so that $G = \langle S, R\,|\, S^2, R^3\rangle$. We observe that the translations $T_1:= SR$ and $T_2 := SR^2$ are of infinite order in $G$ and generate a free semigroup $G_0$ in $G$. Every element of $G$ can be written uniquely as either $w$, $Sw$, $wS$ or $SwS$, where $w \in G_0$. Now the Rademacher quasimorphism $f$ on $G$ can be described as follows (see \cite{BargeGhys} and also \cite[Cor. 4.3]{Rolli}): Given $g \in G$, let $w$ be the element in $G_0$ such that $g \in \{w, Sw, wS, SwS\}$. Then $f(g)$ is the number of $T_1$s in $w$ minus the number of $T_2$s in $w$. An embedding of $G$ into the plane is depicted in Figure \ref{FigureRademacher}. (For better readability we have actually drawn a piece of the Caley graph of $G$ with respect to the generating set $\{S, R, R^2\}$; note however, that by continuing the pattern we will not obtain an embedding of the Cayley graph, but only of $G$, since edges will intersect already at the next step.) We claim that the action of $G$ on this embedding is by quasi-automorphisms. Indeed, one immediately reduces to showing that $G_0$ acts by quasi-automorphisms. However, in the above embedding of the Cayley graph, $T_1$ acts by increasing the $x$-coordinate by $1$, while $T_2$ acts by decreasing the $x$-coordinate by $1$, whence $G_0$ even preserves the relative height function. To see that the resulting quasimorphism is indeed the Rademacher quasimorphism, just observe that every $g \in G$ with $f(g) > 5$ maps every point in the Cayley graph to the right and consequently the induced order on $G$ is sandwiched by the Rademacher quasimorphism.
\end{example}
\begin{figure}
	\centering
  \includegraphics[angle={90}]{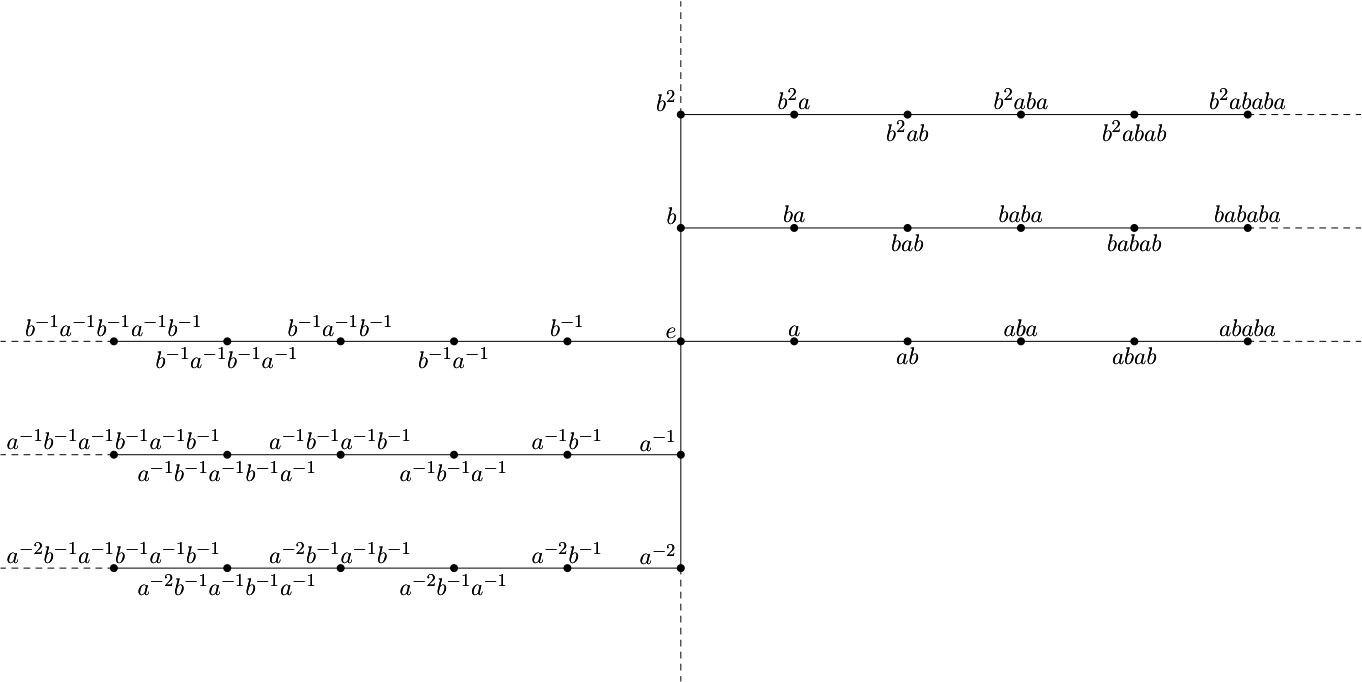}
	\caption{A planar embedding of a subset of the free group}
	\label{FigureBrooks1}
\end{figure}
\begin{figure}
	\centering
  \includegraphics[angle={90}]{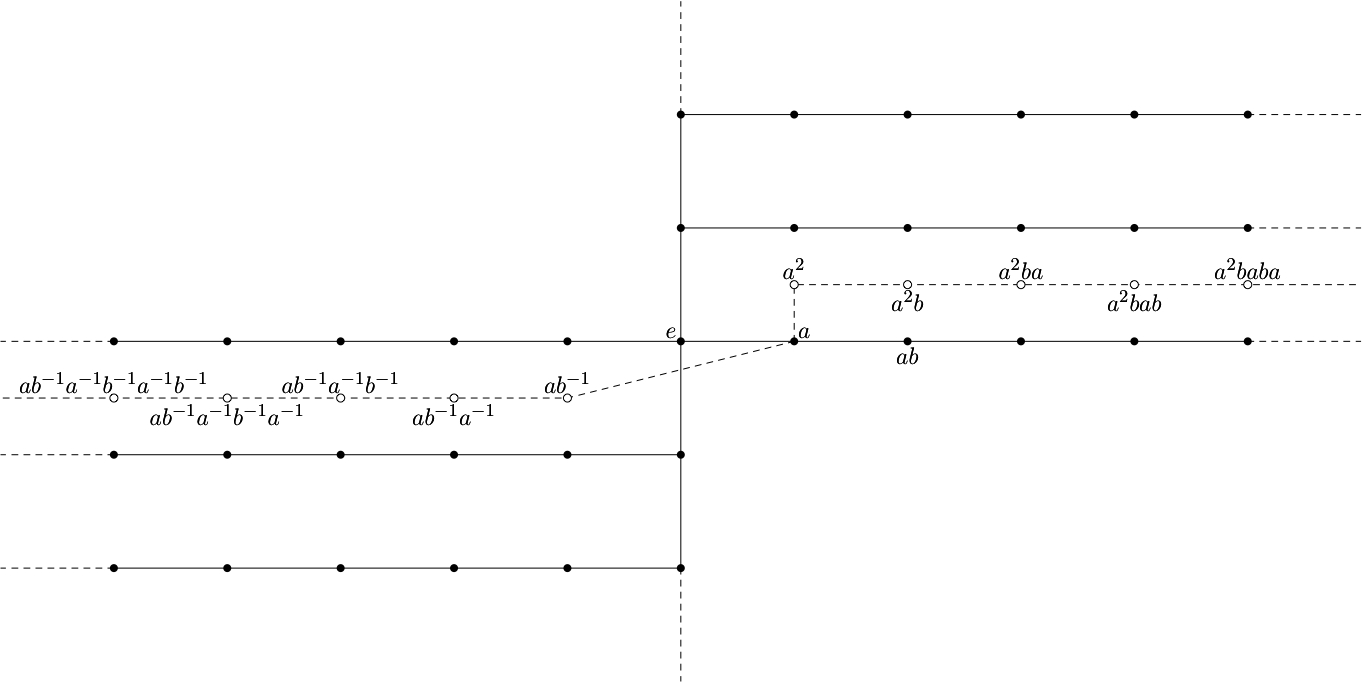}
	\caption{Growing two hairs at $a$}
	\label{FigureBrooks2}
\end{figure}
\begin{example}[Brooks quasimorphism]
We construct an embedding of the free group $G$ on two generators $a$ and $b$ into the plane. We first label the intersection of the lattice $\Z^2$ with the first and third quadrant as in
Figure \ref{FigureBrooks1}, thereby embedding a subset of $G$ into the plane; in a second step we will extend this embedding to a planar embedding of the whole group $G$ by \emph{growing hair}. To explain this procedure, let $w$ be a vertex in the graph in Figure \ref{FigureBrooks1} with the property that at least one of its four neighbours in the Cayley graph of $G$ with respect to $\{a^{\pm 1}, b^{\pm 1}\}$ does not yet appear. We will then add the missing neighbour(s) and some further vertices by the following rules:
\begin{itemize}
\item Assume that the $a$-neighbour $wa$ is missing; this can only happen if the last letter of $w$ is either $a$ or $b^{-1}$. We then add the vertices $wa$, $wab$, $waba$, $wabab$ etc. to the graph. Where precisely we place the first new vertex $wa$ depends on the last letter of $w$: If it is an $a$, then we place $wa$ above $w$ (at a height that has not yet been taken); if it is a $b^{-1}$ we place it two to the right of $w$ at a yet available height. Once $wa$ has been placed, we add $wab$ one to the right of $wa$ at the same height, $waba$ one to the right of $wab$ at the same height etc. For $w = a$ this is depicted in Figure \ref{FigureBrooks2}.
\item Similar rules apply to the other types of missing neighbours: If the $b$-neighbour of $w$ is missing, then the last letter of $w$ is either $b$ or $a^{-1}$. In both cases we place $wb$ above $w$ and add $wba$, $wbab$, $wbaba$, etc. to the right. If the $a^{-1}$-neighbour of $w$ is missing, then the last letter of $w$ is $b$ or $a^{-1}$. In both cases we place $wb$ above $w$ and add $wa^{-1}$, $wa^{-1}b^{-1}$, $wa^{-1}b^{-1}a^{-1}$, etc. to the left. Finally, if the $b^{-1}$-neighbour of $w$ is missing, then the last letter of $w$ is either $a$ or $b^{-1}$ and we place $wb^{-1}$ two steps left below, respectively straight below $w$ accordingly. We then add $wb^{-1}a^{-1}$, $wb^{-1}a^{-1}b^{-1}$ etc. to the left. (See again Figure \ref{FigureBrooks2}.)
\end{itemize}
After applying this procedure once, every vertex in the original embedding has four neighbours, but the newly added vertices have only two neighbours each; we thus continue by growing hair to them according to the same rules. Repeating this procedure ad infinitum we finally obtain an embedding of $G$ into the plane. Similarly as in the last example it can be checked that the action of $G$ on this embedding is unbounded by quasi-automorphisms using the following key observation: If a word $w \in G$ contains $ab$, respectively $b^{-1}a^{-1}$ as a subword $n_+$, respectively $n_-$-times, then the action of $w$ is at uniformly bounded height-distance from a translation by $2(n_+ - n_-)$. This fact can be used to show not only that the action is by quasi-automorphisms, but also that the quasimorphism corresponding to the embedding is given (up to a multiple) by the Brooks quasimorphism associated with the word $ab$, which assigns to $w$ as above the difference $n_+-n_-$. Many other Brooks type quasimorphism admit similar realizations.
\end{example}

\subsection{Incomplete quasi-total triples}\label{SecIncomp}
In our construction of quasi-total orders from quasi-total triples we have always assumed that the quasi-total triples in question were complete and the corresponding $G$-actions were effective. Let us point out that these assumptions can be weakened. We first consider completeness of a quasi-total triple $(X, \preceq, T)$. Let us first observe that $T$ is automatically fixed point-free and non-decreasing, i.e.
  \begin{eqnarray*}
 \forall a \in X \; \forall m \in \mathbb N: \; T^{m}a \not \preceq a.
 \end{eqnarray*}
In general, however, a $T$ need not be strictly increasing. This defect can be repaired as follows: Define a new partial order $\preceq_T$ by setting
 \[
 a \preceq_T b :\Leftrightarrow \exists k \geq 0:\, T^ka \preceq b. 
 \]
Then $T$ is strictly increasing with respect to $\preceq_T$, hence $(X, \preceq_T, T)$ is a complete quasi-total triple. We refer to $(X, \preceq_T, T)$ as the \emph{completion} of $(X, \preceq, T)$. 
The following simple observation explains why the passage from an incomplete to a complete quasi-total triple does not affect the corresponding quasimorphisms.
\begin{proposition} A quasi-total triple $(X, \leq, T)$ and its completion $(X, \leq_T, T)$ define the same height function on $X$, hence give rise to the same translation number.
\end{proposition}
\begin{proof} This follows from
 \begin{eqnarray*}
 h_{\leq_T, T}(a,b) &=& \inf\{m \in \Z\,|\, \exists k \in \mathbb N_0: T^{m-k} b \geq a\}\\
 &=&  \inf\{m \in \Z\,|\, T^{m} b \geq a\} = h_{\leq, T}(a,b).
 \end{eqnarray*}
\end{proof}
Note that if $G$ acts dominatingly on a quasi-total triple, then it also acts dominatingly on its completion.\\

By our definition a quasi-total order is induced by an \emph{effective} dominating $G$-action. Effectiveness is required to make sure that the induced relation on $G$ is indeed a partial order. However, if $G$ acts dominatingly, but not necessarily effectively on some quasi-total triple $(X, \preceq, T)$, then we obtain an effective dominating diagonal action of $G$ on the quasi-total order $(G \times X, \leq, T')$, where $(g, x) < (g', x')$ iff $x < x'$ and $T'(g,x) := (g, Tx)$. Combining these two observations we obtain:
\begin{corollary} If a group acts dominatingly on some quasi-total triple $(X, \preceq, T)$, then it acs dominatingly and effectively on some complete quasi-total triple $(X', \preceq', T')$.
\end{corollary}
Combining this with Theorem \ref{MainResult} and Proposition \ref{Taut} we obtain:
\begin{corollary} 
A group $G$ admits a non-zero homogeneous quasimorphism if and only if if acts dominatingly on some quasi-total triple $(X, \preceq, T)$.
\end{corollary}

\subsection{Admissible bi-invariant total orders}\label{SecTotal}
In this subsection we will provide a proof of Proposition \ref{TotalIntro}. We start by observing:
\begin{lemma}
Let $\leq$ be an admissible bi-invariant total order on a group $G$. Then $\leq$ is a special quasi-total order.
\end{lemma}
\begin{proof} Choose $h \in G^{++}$ and set 
\begin{eqnarray*}
&(X, \preceq, T) := (G, \leq, \rho_h),&
\end{eqnarray*}
where $\rho_h$ denotes right-multiplication by $h$. Then it is easy to check that $(X, \preceq, T)$ is a quasi-total triple and that the induced quasi-total order coincides with $\leq$.
\end{proof}
It thus remains to establish the following result:
\begin{proposition}\label{Total}
Let $\leq$ be an admissible bi-invariant total order on a group $G$. Then the growth functions of $\leq$ are homomorphisms.
\end{proposition}
For the proof we remind the reader of the following simple fact:
\begin{lemma}\label{MultiplyInequalities} Let $G$ be a group and $\leq$ be a bi-invariant partial order on $G$. Then for all $f_1, f_2, g_1, g_2 \in \Gamma$ we have
\[f_1 \geq g_1, f_2 \geq g_2 \Rightarrow f_1f_2 \geq g_1 g_2.\]
\end{lemma}
\begin{proof}[Proof of Proposition \ref{Total}]
We fix $g \in G^{++}$ and show that $\gamma_g$ is a homomorphism. For this let $a, b \in G$. We may assume without loss of generality that $ab \leq ba$. From bi-invariance we then obtain for all $w_1, w_2 \in G$ the inequality
\begin{eqnarray}\label{TotalOrderTrick}
w_{1}abw_{2} &\leq& w_{1}baw_{2}.
\end{eqnarray}
We claim that this implies that for every $n \in \mathbb N$,
\begin{equation}\label{TotalOrderMain}a^{n}b^{n} \leq (ab)^{n} \leq b^{n}a^{n}.\end{equation}
Indeed, using \eqref{TotalOrderTrick} repeatedly we obtain
\begin{eqnarray*}
(ab)^{n}&=&abab \cdot \dots \cdot ab\geq a^{2}bbab\cdot...\cdot ab\\
&\geq& a^{2}babb\cdot...\cdot ab \geq a^{3}b^{3}ab\cdot...\cdot ab\\
&\geq& a^{n}b^{n},
\end{eqnarray*}
and the other inequality is proved similarly. If we abbreviate
\[\gamma_n(g,h) := \inf\{p \in \Z\,|\, g^p \geq h^n\}\quad(h \in G),\]
then we obtain
 \begin{eqnarray}\label{TotalOrderEstimate}
 \gamma_{1}(g,a^{n}b^{n})\leq\gamma_{n}(g,ab)\leq \gamma_{1}(g,b^{n}a^{n}).\end{eqnarray}
On the other hand totality of $\leq$ yields for every $n \in \mathbb N$,
\[g^{\gamma_n(g,a)-1} \leq a^{n} \leq  g^{\gamma_n(g,a)}, \quad g^{\gamma_n(g,b)-1} \leq b^{n} \leq  g^{\gamma_n(g,b)}\] 
and thus by Lemma \ref{MultiplyInequalities}
\begin{eqnarray*}
g^{\gamma_n(g,a)+\gamma_n(g,b)-2} \leq a^{n}b^n \leq  g^{\gamma_n(g,a)+\gamma_n(g,b)+2}, \\ g^{\gamma_n(g,a)+\gamma_n(g,b)-2} \leq b^na^n \leq  g^{\gamma_n(g,a)+\gamma_n(g,b)+2}.\end{eqnarray*}
We deduce that
\begin{eqnarray*}
|\gamma_{1}(g,a^{n}b^{n})-\gamma_{n}(g,a)-\gamma_{n}(g,b)|\leq 2\\
|\gamma_{1}(g,b^{n}a^{n})-\gamma_{n}(g,a)-\gamma_{n}(g,b)|\leq 2.
 \end{eqnarray*}
Combining this with \eqref{TotalOrderEstimate} we obtain
\[\gamma_{n}(g,a)+\gamma_{n}(g,b)-2 \leq \gamma_{n}(g,ab) \leq \gamma_{n}(g,a)+\gamma_{n}(g,b)+2\]
Dividing by $n$ and passing to the limit $n \to \infty$ we get  \[\gamma(g,a)+\gamma(g,b) \leq \gamma(g,ab) \leq \gamma(g,a)+\gamma(g,b).\] This shows that $\gamma_g$ is a homomorphism.
\end{proof}
The condition of admitting a bi-invariant total order is rather restrictive. We refer the reader to \cite{Navas} and the references therein for various characterizations and properties of totally bi-orderable groups. Here we just consider two classes of examples given by \emph{free groups} and \emph{pure braid groups} respectively. We refer the reader to \cite[Sec. 7.2]{BraidGroups} for background on their bi-invariant total orders.
\begin{example}
Consider first the case of a free group $F_n$ on $n$ free generators $S = \{x_1, \dots, x_n\}$. The natural total order $x_1 < \dots < x_n$ on $S$ then induces a total bi-invariant order on $F_n$ via Magnus expansion, see \cite[Prop. 7.11]{BraidGroups}. By Proposition \ref{Total} the associated growth functions on $F_n$ are homomorphisms, and it is easy to see that up to normalization they are given by the counting homomorphism $\mu_{x_n}$ associated with $x_n$, i.e. if $w=s_{1} \cdots s_{m}$ with $s_{i} \in S\cup S^{-1}$ then
\[\mu_{x_n}(w) = \#\{i \,|\, s_i = x_n\} -\#\{j \,|\, s_j = x_n^{-1}\}.\]
\end{example}
\begin{example}
Following the notation in \cite{BraidGroups} we now denote by $P_n$ the pure braid group on $n$ strands (i.e. the kernel of the natural surjection of the braid group $B_n$ onto the symmetric group on $n$ letters) and by $A_{ij}$, $1 \leq i <j \leq n$ its canonical generators. We also denote by $U_n$ the free group on generators $A_{j, n}$, $1 \leq j < n$ and equip it with the bi-invariant total order induced from the order $A_{1,n} < \dots  < A_{n-1, n}$ on generators. Then we can define inductively a bi-invariant total order on $P_n$ by demanding that the morphisms in the short exact sequence 
\[1 \to U_n \to P_n \to P_{n-1} \to 1\]
are order-preserving and the order on $P_2 \cong \Z$ is the standard one \cite[p. 281]{BraidGroups}. Again the associated growth functions are easy to compute; they coincide up to normalization with the iterated projection
\[\pi_n: P_n \to P_{n-1} \to \dots \to P_2 \cong \Z \subset \R.\]
To describe $\pi_n$ in terms of generators and relations, observe that the counting homomorphism $\mu_{A_{12}}$ on the free group $F$ on generators $A_{ij}$, $1 \leq i <j \leq n$ descends to a homomorphism $P_n$, and this homomorphism coincides with $\pi_n$. Thus $\pi_n$ counts the occurences of $A_{12}$ in a given word in the pure braid group.
\end{example}

\section{Total triples and circular quasimorphisms}\label{SecTotal}\label{SecCircle}
\subsection{Total triples}
A very special class of examples of quasi-total triples $(X, \preceq, T)$ is given by \emph{totally ordered spaces} $(X, \preceq)$ together with a dominating automorphism $T$. We then say that $(X, \preceq T)$ is a \emph{total triple}. In this situation the theory simplifies considerably. For instance, the height function admits the following simpler description:
\begin{proposition}
Let $(X, \preceq, T)$ denote a complete total triple and let $a,b \in X$. Then $h_T(a,b)$ is the unique integer such that 
\begin{eqnarray*}
T^{h_T(a,b)-1}.b \prec a \preceq T^{h_T(a,b)}.b.
\end{eqnarray*}
\end{proposition}
Now let us specialize further to the case where $X$ coincides with $G$. In this case $\preceq$ is a left-invariant order on $G$ and we have a distinguished basepoint given by $a = e$. Given $g \in G$ define $n := n(g)$ to be the unique integer satisfying
\[T^{n-1}.e \preceq g \preceq T^n.e;\]
Then, as a special case of the last proposition we see that the function $g \mapsto n(g)$ is at bounded distance from the translation number $T_{G, \preceq, T}$ associated with $(G, \preceq, T)$. From this description we see in particular that our construction generalizes a construction of Ito \cite{Ito}:
\begin{corollary}\label{ItoComparison}
Let $G$ be a group, $\preceq$ a left-invariant total order on $G$, $x \in G$ and $\rho_x(g) := gx$. Assume that $(G, \preceq, \rho_x)$ is a total triple. Then the translation number $T_{G, \preceq, \rho_x}$ is the homogeneization of the quasimorphism  $\rho_{x, \preceq}^G$ constructed in \cite{Ito}.
\end{corollary}
A particular example seems worth mentioning at this point:
\begin{example} Let $B_n$ be again the $n$-string braid group and denote by $\sigma_{1}, \dots, \sigma_{n-1}$ its canonical (Artin) generators. There is a canonical left-ordering $\preceq$ on $B_n$, which is described e.g. in \cite{Braids} and sometimes called the \emph{Dehornoy order}. If we choose
\[x := ((\sigma_1 \cdots \sigma_{n-1})(\sigma_1\sigma_2)\sigma_1)^2,\]
then $x$ is central in $B_n$ and $(B_n, \preceq, \rho_x)$ is a total triple. Combining the last corollary with \cite[Example 1]{Ito}, we see that the translation number $T_{B_n, \preceq, \rho_x}$ is the homogeneization of the Dehornoy floor quasimorphism.
\end{example}
In \cite{Ito} it is always assumed that $T=\rho_x$ for some $x \in G$. If $G$ is assumed countable, then this is not a serious restriction:
\begin{lemma}\label{GExtension}
Let $G$ be a countable group and $(G, \preceq, T)$ be a total triple with a dominating $G$-action. Then there exists a supergroup $G_1$ of $G$, a total order $\preceq_1$ on $G_1$ and an element $x \in G_1$ with the following properties:
\begin{itemize}
\item[(i)] $\preceq_1$ is a left-invariant, total order on $G_1$ and $\preceq_1|_G = \preceq$.
\item[(ii)] $x \in Z(G_1)$ and $x$ is dominant for $\preceq_1$.
\item[(iii)] $(G_1, \preceq_1, \rho_x)$ is a total triple with a dominating $G_1$-action.
\end{itemize}
Moreover, $G_1$ is isomorphic to a quotient of $G \times \Z$ and the embeddings of $G$ into $G \times \Z$ and $G_1$ are compatible.
\end{lemma}
\begin{proof} Let $G_1$ be the subgroup of $\rm{Aut}(G, \preceq, T)$ generated by $G$ and $T$ and set $x := T$. Here $G$ acts on itself by left-multiplication. Since $G$ and $T$ commute, this group is a quotient of $G \times \Z$ and $x$ is central. Note that $G_1$ acts on $G$ preserving $\preceq$. To define $\preceq_1$ choose an enumeration $\{g_i\}_{i \in \mathbb N}$ of $G$ with $g_1 = e$; then define that $g \preceq_1 h$ if and only if $(gg_i) \preceq(hg_i)$ with respect to the lexicographical order on $G^{\mathbb N}$. Since $G_1$ acts effectively on $G$, this defines a total order and $x$ is dominant, since $T$ is dominant. Also, $\preceq_1$ is $G_1$-invariant, since $\preceq$ is. Finally, let $g, h \in G$ be distinct; then either $g \prec h$ or $g \succ h$. In the former case we have
$g.g_1 \prec h.g_1$ (since $g_1 = e$) and thus $g \prec_1 h$, while in the second case we have $g \succ_1 h$. This shows that $\preceq_1$ restricts to $\preceq$ on $G$. 
\end{proof}
Thus in studying total triples $(G, \preceq, T)$ over a countable group $G$ we may focus on the case, where $T = \rho_x$ for a central dominant $x \in G$.
\subsection{From circular quasimorphisms to total triples}
In this subsection we study quasimorphisms which arise from lifts of actions on the circle:
\begin{definition}
Let $G$ be a group. A nonzero homogeneous quasimorphism $f$ on $G$ is called \emph{circular} if there exists an injective homomorphism $\phi:G \to {\rm Homeo}^+_\Z(\R)$ such that $f = \phi^*T_\R$.
\end{definition}
It turns out that circular quasimorphisms are closely related to total triples. The precise relation is somewhat technical, and we offer three different (essentially equivalent) formulations:
\begin{proposition}\label{EasyDir1} \label{EasyDirection}\label{EasyDir2}
Let $G$ be a group, and $f$ be a circular homogeneous quasimorphism on $G$.
\begin{itemize}
\item[(i)] There exists a left-invariant total order $\preceq$ on $G$ such that the growth functions of the order induced from $\preceq$ via the left-action of $G$ on itself are multiples of $f$.
\item[(ii)] There exists a quasi-total triple $(G, \preceq_0, T)$ realizing $f$ with the property that $\preceq_0$ can be refined into a left-invariant total order $\preceq$ on $G$.
\item[(iii)] Assume that $f$ is unbounded on the center of $G$. Then there exists a total triple $(G, \preceq, T)$ realizing $f$.
\end{itemize}
\end{proposition}
\begin{proof} (i) We first recall \cite{Navas} that every enumeration $\{q_n\}$ of $\mathbb Q$ defines a total order $\preceq$ on $H :=  {\rm Homeo}^+_\Z(\R)$ by setting  $g \preceq h$ if and only if $(gq_n) \leq (hq_n)$ with respect to the lexicographic ordering on $\R^{\mathbb N}$. Indeed, this follows from the fact that every $h \in H$  is uniquely determined by its restriction to $\mathbb Q$. We fix such an enumeration and the corresponding ordering $\preceq$ once and for all. By construction, $\preceq$ is left-invariant. Denote by $\leq_H$ the bi-invariant order on $H$ induced by $\preceq$. Then $\leq_H$ is sandwiched by $T_{\R}$. Indeed, assume $T_{\R}(h) > 10$. Then for all $q \in \R$ we have $h.q > q$, whence $(hq_n) \succ (eq_n)$ and thus $h \geq_H e$.\\

Now assume $f: G\to \R$ is circular and nonzero, say $f = \phi^*T_{\R}$ for some injection $\phi:G \to H$. For notation's sake let us assume that $G$ is a subgroup of $H$ and $\phi$ the inclusion. Then the restriction $\preceq|_G$ defines a left-invariant total order on $G$. Let $\leq$ be the order on $G$ induced by $\preceq|_G$. From the fact that $T_{\R}$ sandwiches $\leq$ we deduce that $f$ sandwiches $f^*\leq_H$; since $\leq$ is a refinement of $f^*\leq_H$, it also sandwiches $\leq$.\\

(ii) Argue as in (i), but define $\preceq_0$ to be the bi-invariant order induced by $\preceq$ and choose $T$ to be right multiplication by some element $g \in G$ with $\phi(g) > 10$.\\

(iii) Construct $\preceq$ as in (i) and choose $T$ to be multiplication by a central element $x$ with $f(x) > 10D(f)+5$.
\end{proof}
For countable groups we will establish a partial converse to Proposition \ref{EasyDir2} in Theorem \ref{Dynamical}  below. 

\subsection{From total triples to circular quasimorphisms}
The goal of this section is to establish the following partial converse of Proposition \ref{EasyDir2}:
\begin{theorem}\label{Dynamical} Let $G$ be a countable group and $(G, \preceq, T)$ be a total triple with a dominating $G$-action. Denote by $\leq$ the induced bi-invariant order on $G$. Then the growth functions of $\leq$ are nonzero circular quasimorphisms.
\end{theorem}
We will first establish the theorem under the additional hypothesis that $T = \rho_x$ for some central dominant $x \in G$. In a second step we will then reduce the general case to this case by means of Lemma \ref{GExtension}. The first step of the proof uses crucially the notion of a dynamical realization \cite{Navas}:
\begin{definition}
Let $\preceq$ be a left-invariant total order on $G$. A \emph{dynamical realization} of $\preceq$ is a pair $(\phi, t)$ consisting of an injective homomorphism $\phi: G \to {\rm Homeo}^+(\R)$ and a $\phi$-equivariant embedding $t: G \to \R$ such that $t(e) = 0$, $\inf_{g \in G} t(g) = -\infty$, $\sup_{g \in G} t(g) = \infty$ and
\begin{eqnarray}\label{DynReal}
f \prec g \Leftrightarrow \phi(f).0 < \phi(g).0.\end{eqnarray}
A \emph{dynamical realization} of $\preceq$ is \emph{special} if $\phi(G)$ centralizes the translation $T: x \mapsto x+1$; it is called \emph{adapted} to $x \in G$ if $\phi(x) = T$.
\end{definition}
The following is well-known:
\begin{proposition}
Let $G$ be a countable group and $\preceq$ a left-invariant total order on $G$.
\begin{itemize}
\item[(i)] $\preceq$ admits a dynamical realization. 
\item[(ii)] There exists a special dynamical realization of $\preceq$ adapted to $x \in G$ if and only if $x$ is both central and dominant for $\preceq$.
\end{itemize}
\end{proposition}
\begin{proof} (i) see \cite[Prop. 2.1]{Navas}. (ii) Assume that such a realization exists. Since $\phi(x)$ is contained in the centralizer of $\phi(G)$ and $\phi$ is injective we must have $x \in Z(G)$. Also, given any $g \in G$ we find $n \in \mathbb N$ with $\phi(g).0 < n =\phi(x^n).0$. We deduce that $g \preceq x^n$, which shows that $x$ is a dominant. Thus the conditions are necessary. On the other hand, assume that $x \in Z(G)$ is dominant. Then every element in $G$ may be written uniquely as $g = g_0x^n$ with $n \in \mathbb Z$ and $e \preceq g_0 \preceq x$. Now define the embedding $t$ as follows: Set $t(e) = 0$, $t(x) = 1$ and let $\{g_k\}_{k \in \mathbb N}$ be an enumeration of the order interval $[e, x]$ with $g_1 = e$, $g_2 = x$. Inductively assume $t(g_1), \dots, t(g_{i-1})$ have been defined. Then there exists $g_m, g_M$ such that $g_m < g_i < g_M$ and $]g_m, g_M[ \cap \{g_1, \dots, g_{i-1}\} = \emptyset$. We then define $t(g_i) := (t(g_m)+t(g_M))/2$. Now extend the map $t:[e, x] \to \R$ to all of $G$ by the formula $t(g_0x^n) = n+t(g_0)$. The action of $G$ on $t(G)$ given by $\phi(g)t(h) := t(gh)$ extends continuously to the closure of $t(G)$ and can be extended to an action of homeomorphisms on $\R$ in a standard way, see \cite{Navas}. We have $\phi(x).t(g) = t(g+1)$, hence $\phi(x) = T$ on $\overline{t(G)}$. From the construction of the extension in loc. cit. we deduce $\phi(x) = T$ on all of $\R$.
\end{proof}
We now fix a total triple of the form $(G, \preceq, \rho_x)$ with $x \in Z(G)$ dominant and a dynamical realization $(\phi, t)$ of $\preceq$ adapted to $x$. As before, we denote by $\leq$ the order induced by $\preceq$ on $G$. We recall that $\leq$ is admissible and that its growth functions are multiples of $T_{(G, \preceq, \rho_x)}(g)$. We now aim to describe these growth functions in terms of the homomorphism $\phi: G \to {\rm Homeo}^+_{\Z}(\R)$. To this end we observe that $\phi$ allows us to pullback the classical translation number $T_\R$ to a homogeneous quasimorphism $\phi^*T_\R$ on $G$. The following was observed in \cite{Ito}:
\begin{proposition}[Ito]
Let $G$ be a countable group, $\preceq$ a left-invariant total order on $G$ and $x \in G$ a central dominant. Let $(\phi, t)$ be a dynamical realization of $\preceq$ adapted to $x$. Then \[T_{(G, \preceq, \rho_x)}- \phi^*T_\R: G \to \R\]
is a homomorphism.
\end{proposition}
\begin{proof} By the proof of {\cite[Theorem 3]{Ito}} the pullback of the bounded Euler class $-e_b := dT_\R$ under $\phi$ in real bounded cohomology is the class represented by the differential of the quasimorphism denoted $\rho_{x, \preceq}^G$ in \cite{Ito}. (In fact, this is even true for the corresponding integral bounded cohomology classes, but we do not need this stronger statement here.) Since $T_{G, \preceq, \rho_x}$ is at bounded distance from $\rho_{x, \preceq}^G$ by Corollary \ref{ItoComparison}, we deduce that the differential of $f:=T_{G, \preceq, \rho_x}- \phi^*T_\R$ represents the trivial class in $H^2_b(G; \R)$. Now $f$ is both homogeneous and cohomologically trivial, hence a homomorphism.
\end{proof}
We will strengthen this as follows:
\begin{lemma}\label{ItoImproved}
Let $G$ be a countable group, $\preceq$ a left-invariant total order on $G$ and $x \in G$ a central dominant. Let $(\phi, t)$ be a dynamical realization of $\preceq$ adapted to $x$. Then \[T_{(G, \preceq, \rho_x)} = \phi^*T_\R.\]
\end{lemma}
\begin{proof} Since both quasimorphisms are homogeneous it suffices to show that they are at bounded distance. For this we may replace $T_\R$ by the function $g \mapsto \phi(g).0$ and $T_{(G, \preceq, \rho_x)}(g)$ by $h_T(g, e)$, since those are at bounded distance from the original functions. Now choose $n$ so that
\[\phi(x)^{n-1}.0 = n-1 < \phi(g).0 \leq n = \phi(x)^n.0.\]
This implies both $|\phi(g).0 - n| < 1$ and $x^{n-1} \prec g \preceq x^{n+1}$, the latter by \eqref{DynReal}. We may rewrite the last chain of inequalities by
\[\rho(x)^{n-1}.e= x^{n-1} \prec g \prec x^{n+1} = \rho(x)^{n+1}.e.\]
From this we deduce that $|h_T(g,e)-n| < 2$, whence $|\phi(g).0 - h_T(g,e)| < 3$. 
\end{proof}
Now we can deduce the theorem:
\begin{proof}[Proof of Theorem \ref{Dynamical}] Let $(G, \preceq, T)$ be any total triple with a dominating $G$-action. We then construct the extended triple $(G_1, \preceq_1, \rho_x)$ as in Lemma \ref{GExtension} and denote by $\leq_1$ the order induced by $\preceq_1$ on $G_1$. We then choose a dynamical realization $(\phi_1, t_1)$ of $\preceq_1$ adapted to $x$ and
deduce from Theorem \ref{MainResult}  and Lemma \ref{ItoImproved} that $\leq_1$ is sandwiched by $T_{G_1, \preceq_1, \rho_x} = \phi_1^*T_\R$. We thus find a constant $C$ such that for $g \in G_1$ with $T_\R(\phi_1(g)) > C$ we have
\begin{eqnarray}\label{DynamicalReduce}
\forall x \in G_1: \, g.x \succeq_1 x.
\end{eqnarray}
Denote by $\leq$ the bi-invariant order induced by $\preceq$ on $G$ and by $\phi$ the composition of the inclusion $G \to G_1$ with $\phi_1$. We then claim that $\leq$ is sandwiched by $\phi^*T_\R$. Indeed, assume $g \in G$ satisfies $\phi^*T_\R(g) > C$; then \eqref{DynamicalReduce} holds, and in particular 
\begin{eqnarray*}\label{DynamicalReduce2}
\forall x \in G: \, g.x \succeq_1 x.
\end{eqnarray*}
But since $\succeq_1|_{G} = \succeq$, this shows that $g \geq e$, which yields the desired sandwiching result.
\end{proof}
For quasimorphisms which are unbounded on the center of $G$ we have obtained a complete characterization of circularity:
\begin{corollary}
Let $f: G \to \R$ be a quasimorphism, which is unbounded on the center of $G$. Then the following are equivalent:
\begin{itemize}
\item[(i)] $f$ is circular.
\item[(ii)] $f$ can be realized by a total triple $(G, \preceq, T)$.
\end{itemize} 
\end{corollary}
For quasimorphism, which are not unbounded on the center of $G$ the situation is slightly more technical, as witnessed by the more complicated formulation of Proposition \ref{EasyDir2} for such quasimorphisms. 

\section{Smooth quasi-total triples from causal coverings}\label{SecSmoothCase}
\subsection{Causal coverings}
In this section we study quasi-total triples induced by smooth partial orders on manifolds. The notion of smooth partial orders that we use here is discussed in the appendix. From now on we will denote by $(\widetilde{M}, \widetilde{\mathcal C})$ a causal manifold in the sense of Definition \ref{DefCausalManifold} and by $\preceq$ the associated partial order. We also denote by $G(\widetilde{M}, \widetilde{\mathcal C})$ the associated automorphism group (see Definition \ref{AutMC}). The main problem of this section can then be formulated as follows:
\begin{problem}\label{ProblemSmooth}
Given a causal manifold $(\widetilde{M}, \widetilde{\mathcal C})$, is there an automorphism $T \in G(\widetilde{M}, \widetilde{\mathcal C})$, which turns $(\widetilde{M}, \preceq, T)$ into a quasi-total triple?
\end{problem}
\begin{remark} For the purpose of this subsection we could as well consider a weakly causal manifold in the sense of Definition \ref{DefCausalManifold} and study the associated strict causality $\preceq_s$ instead of $\preceq$. We would then ask for an automorphism $T$ of   $(\widetilde{M}, \widetilde{\mathcal C})$ turning $(\widetilde{M}, \preceq_{s}, T)$) into a quasi-total triple. All results of this subsection remain valid in this setting; the difference between $\preceq_s$ and $\preceq$ will only become important when we discuss global hyperbolicity in Subsection \ref{SecHyperbolicity} below.
\end{remark}
We now fix a causal manifold $(\widetilde{M}, \widetilde{\mathcal C})$. We oberve that if $T$ as in Problem \ref{ProblemSmooth} exists, then it has to be of infinite order. Hence, fix $T \in {\rm Aut}(\widetilde{M}, \widetilde{\mathcal C})$ of infinte order and assume moreover that the group $\Gamma \cong \Z$ generated by $T$ in $G(\widetilde{M}, \widetilde{\mathcal C})$ acts properly discontinuously on $\widetilde{M}$; then $ M := \Gamma\backslash\widetilde{M}$ is a manifold and
\[p: \widetilde{M} \to M\]
is a covering projection. We will denote by $\check G$ the centralizer of $T$ (hence $\Gamma$) in $G(\widetilde{M}, \widetilde{\mathcal C})$. Then $G :=\check G/\Gamma$ acts on $M$ and $\check G$ is a central extension of $G$ by $\Gamma$. We refer to the central extension
\[0 \to \Gamma \to \check G \to G \to 1\]
as the central extension \emph{associated} with the covering $p: \widetilde{M} \to M$.
\begin{definition}
The covering $p: \widetilde{M} \to M$ is called a \emph{causal covering} if $M$ is totally acausal (in the sense of Definition \ref{DefCausalManifold}).
\end{definition}
\begin{lemma}
Assume that $p: \widetilde{M} \to M$ is a causal covering, Then either $T$ or $T^{-1}$ is dominant.
\end{lemma}
\begin{proof} Denote by $p: \widetilde{M} \to M$ the covering projection and let $a,b \in \widetilde{M}$. Since $M$ is totally acausal there exists a closed causal loop $\gamma_{p(a), p(b)}$ at $p(a)$ through $p(b)$. We can lift this loop to a curve $\gamma_{a,b}$ with initial point $a$; the result is a causal curve through $a$ and some $T$-translate of $b$. Thereby we find integers $l(a,b) \in \Z$ with $a \preceq T^{l(a,b)} b$. Since $\widetilde{M}$ does not contain causal loops, we may assume $l(a,a) > 0$ for some given basepoint $a$ upon possibly replacing $T$ by its inverse. We claim that this implies ${l(b,b)} > 0$ for all $b$. Indeed, suppose otherwise, say $b \succeq T^k b$ with $k > 0$. We then find $m > 0$ with
\[T^{-mk} a \preceq b \preceq T^{mk} a, \quad T^{-mk} b \succeq b \succeq T^{mk} b,\]
hence $b \succeq T^{mk} (T^{-mk} a) = a$ and $b \preceq  T^{-mk}(T^{mk} a) = a$. This yields $a=b$ and $l(a,a) <0$, which is a contradiction. We see in particular that we can choose $l(a,b)$ positive by adding a suitable multiple of $l(b,b)$. This implies that $T$ is dominant.
\end{proof}
Thus replacing $T$ by $T^{-1}$ if necessary we will assume from now on that $T$ is the unique dominant generator of $\Gamma$. We see from the proof of the last proposition that for any pair $a, b \in \widetilde{M}$ there exists $n(a,b) := \min\{l(a,b), l(b,a)\} \in \mathbb N$ such that $a \preceq T^{n(a,b)}b$ or $b \preceq T^{n(a,b)} a$. However, the number $n$ is in general not uniformly bounded. 
\begin{definition}
A causal covering $\widetilde{M} \to M$ is called  \emph{quasi-total} if the number $n(a,b)$ is uniformly bounded.
\end{definition}
Equivalently, $(\widetilde{M}, \preceq, T)$ is a quasi-total triple. In the next section we will provide two different criteria which guarantee this property. Before, let us give some elementary examples of quasi-total causal coverings. Firstly, the classical translation number $T_\R$ is associated with the causal covering $\R \to S^1$. The following example can be considered as a smooth twisting of this trivial example; as in Example \ref{Twisting} it is easy to argue that this sort of twisting produces fundamentally different quasimorphisms.
\begin{example}
Let $\widetilde{M} := \R \times ]-1,1[$ be a strip of bounded diameter with basepoint $x_0 := (0,0)$
and let $C \subset \R^2$ be a closed regular cone which contains the positive $x$-axis in its interior. Then the translation invariant cone field on $\R^2$ modelled on C restricts to a conal structure $\widetilde{\mathcal C}$ on $\widetilde{M}$, and the conal manifold $(\widetilde{M}, \widetilde{\mathcal C})$ is in fact causal, since every non-constant causal curve is strictly monotone in the $x$-coordinate. Since the cone $C$ contains the positive $x$-axis in its interior we find $x_{\pm} \in \R$ such that 
\[x_{\pm} \in \R \times \{\pm 1\} \cap C.\]
Choose $x_{\pm}$ minimal with this property and set $x_0 := 2\max\{x_+, x_-\}$. Let $T$ be the translation along the $x$-axis by $x_0$, i.e. $T(x,y) := (x+x_0, y)$ and let $M := \widetilde{M}/\langle T \rangle$. Then $M \cong S^1 \times ]-1,1[$ and $\widetilde{M} \to M$ is a quasi-total causal covering.
\end{example}
Similar twists can also be defined for the examples from Lie groups as discussed below.

\subsection{Criteria for quasi-totality}
Before we can discuss further examples, we need to develop criteria whoch guarantee quasi-totality. Throughout this section we fix a causal covering $\widetilde{M} \to M$ and denote by
\[0 \to \Gamma \to \check G \to G \to 1\]
the associated central extension of automorphism groups. The easiest way to guarantee quasi-totality is to demand enough transitivity of $G$ on $M$.
\begin{definition}
An action of a group $G$ on a space $X$ is \emph{almost $2$-transitive} it there exists a $G$-orbit $X^{(2)} \subset X^2$ with the property that
\[\forall x,y \in X \; \exists z \in X:\; \{(x,z), (z,y)\} \subset X^{(2)}.\]
In this case we call $X$ \emph{almost $2$-homogeneous} and write $x \pitchfork y$ to indicate that $(x,y) \in X^{(2)}$.
\end{definition}
Then we obtain:
\begin{theorem}\label{CausalQT} Let $p: \widetilde{M} \to M$ be a causal covering. If $G := \check G/\Gamma$ acts almost $2$-transitively on $M$, then $p$ is quasi-total.
\end{theorem}
Note that the almost $2$-transitivity of $G$ on $M$ implies automatically that $M$ is totally acausal.
We prepare the proof of Theorem \ref{CausalQT} by the following lemma:
\begin{lemma} In the situation of Theorem \ref{CausalQT} there exists a constant $N \in \mathbb N$, depending only on $M$, and a point $c \in \widetilde{M}$ such that for all $b \in \widetilde{M}$ there exists $l \in \Z$ such that $b \preceq T^{l}c \preceq T^Nb$.
\end{lemma}
\begin{proof} Let $a \in \widetilde{M}$ be some basepoint and $x := p(a)$. Then we find $z \in M$ with $z \pitchfork x$ and a closed causal loop $\gamma_x: [0,1] \to M$ be a closed causal loop at $x$ with $z = \gamma_x(1/2)$. Let $\hat \gamma_a$ be the lift of $\gamma_x$ with initial point $a$. Then $\hat \gamma_a(0) \preceq \hat \gamma_a(1)$; we thus find $N_0 > 0$ such that $\hat \gamma_a(1) = T^{N_0}\hat \gamma_a(0)$. If we define $c := \hat \gamma_a(1/2)$, then
\begin{eqnarray*}
a \preceq c \preceq T^{N_0} a.
\end{eqnarray*}
Now let $d \in \widetilde{M}$, $w := p(d)$ and assume $w \pitchfork z$. Then we find $g_0 \in G_0$ with $g_0.x = w$ and $g_0.z = z$. Thus $\gamma_w = g_0.\gamma_x$ is a closed causal loop at $w$ through $z$. Now let $\hat \gamma_d$ be a lift of $\gamma_w$ with initial point $d$. Let $g \in G$ be a lift of $g_0$; then $g$ maps $a$ to a point in the fiber of $d$, and modifying $g$ by a deck transformation if necessary we can assume $g.a = d$. Then $g.\hat\gamma_a$ is a lift of $\gamma_w$ with initial point $d$, hence $\hat \gamma_d = g.\hat\gamma_a$ by uniqueness. Now we have 
\[\hat \gamma_d(1) = g \hat\gamma_a(1) = gT^{N_0}\hat\gamma_a(0) = T^{N_0}g\hat\gamma_a(0) = T^{N_0}\hat\gamma_d(0).\]
Now since $\gamma_w(1/2) = z$ we find $l \in \mathbb Z$ such that $\hat\gamma_d(1/2) = T^l c$. We thus have established
 \begin{eqnarray}\label{QTEstimate1}
d \preceq T^{l}c \preceq T^{N_0}d
\end{eqnarray}
under the assumption $p(d) \pitchfork z$. Now consider the case of an arbitrary $b \in \widetilde{M}$ and let $y := p(b)$. We then find $d \in \widetilde{M}$ such that $w := p(d)$ satisfies
\[y \pitchfork w \pitchfork z.\]
Then \eqref{QTEstimate1} holds for some $l \in \Z$. Moreover, we find $h_0 \in G_0$ with $h_0.(x,z) = (y,w)$. Define $\gamma_y := h_0.\gamma_x$ and denote by $\hat \gamma_b$ the lift of $\gamma_y$ with initial point $b$. By the same argument as before we then show $\hat \gamma_b(1) = T^{N_0}\hat\gamma_b(0)$. Now $\gamma_y(1/2) = w$, so we find $l' \in \Z$ with $\hat \gamma_b(1/2) = T^{l'}d$. Thus
 \begin{eqnarray}\label{QTEstimate2}
b \preceq T^{l'}d \preceq T^{N_0}b
\end{eqnarray}
Combining \eqref{QTEstimate1} and \eqref{QTEstimate2} we obtain
 \begin{eqnarray*}\label{QTEstimate3}
b \preceq T^{l'+l}c \preceq T^{2N_0}b
\end{eqnarray*}
We may thus choose $N := 2N_0$.
\end{proof}
Now we deduce:
\begin{proof}[Proof of Theorem \ref{CausalQT}] Let $c$ and $N$ be as in the lemma. Given $a,b \in \widetilde{M}$ we find $l, l' \in \Z$ such that $b \preceq T^lc \preceq T^Nb$ and $a \preceq T^{l'}c\preceq T^N a$. We may assume w.l.o.g. that $l' \geq l$. Now for all $k \geq 1$ we have
\[T^{l'-l}b \preceq T^{l'}c \preceq T^N a \preceq T^{kN}a,\]
hence $T^{l'-l-kN} b \preceq a$. Now choosing $k$ appropriately we can ensure that $-N \preceq l'-l-kN \preceq N$. Thus
\[\forall a,b \in \widetilde{M} \exists k \in \{-N, \dots, N\}: (a \preceq T^kb) \vee (b \preceq T^ka).\]
Now \eqref{QuasiTotal} follows for $N(\widetilde{M}) := 2N$.
\end{proof}
The almost $2$-transitivity condition of Theorem \ref{CausalQT} is rather strong and not always easy to check in practice. We are thus looking for an alternative condition that ensures quasi-totality. One consequence of quasi-totality is that  $T^k.x \succeq x$ for all $x$ and a uniformly bounded $k$. Here we shall assume the slightly stronger condition 
\begin{eqnarray}\label{CompactnessAssumption1}
T.x \succeq x,
\end{eqnarray}
i.e. completeness of the triple $(\widetilde{M}, \preceq, T)$. We then call the covering $p: \widetilde{M} \to M$ a \emph{complete} causal covering. This terminology understood we have the following useful criterion:
\begin{theorem}
Assume that $M$ is compact. Then any complete causal covering $p: \widetilde{M} \to M$ is quasi-total.
\end{theorem}
\begin{proof}  We first claim that there exists point $a, b \in \widetilde{M}$ such that $a \preceq x \preceq b$ for all $x \in U$. Indeed, choosing $a$ and $b$ close enough we can ensure that $\exp({\rm Int}(\widetilde{\mathcal C}_a))$ and $\exp({\rm Int}(- \widetilde{\mathcal C}_b))$ have open intersection. By compactness of $M$ there exists finally many elements $g_1, \dots, g_l$ such that
\begin{eqnarray}\label{CompactTrick}
M = p\left(\bigcup_{j=1}^l g_jU\right).
\end{eqnarray}
Thus if we set 
\[H^{\pm}:= \bigcup_{k \geq 0} T^{\pm k} \left(\bigcup_{j=1}^l g_jU\right),\]
then $\widetilde{M} = H^- \cup H^+$. For $j= 1, \dots, l$ we set $a_j := g_ja$, $b_j := g_jb$. Let $m_{ij}$ be integers such that
\[b_i \leq T^{m_{ij}} a_j\]
and set $N := \max m_{ij}$. Then $x \leq T^N y$ for all $x,y \in \bigcup g_jU$, hence $x \leq T^N y$ for all $x \in H^-$, $y \in H^+$. Now let $x,y \in \widetilde{M}$ be arbitrary. We distinguish three cases:
\begin{itemize}
\item If one of them is in $H^-$ and the other is contained in $H^+$, then $y \preceq T^Ny$ or $y \preceq T^N x$.
\item If none of them is in $H^+$, apply $T$ until the first of them is. We may assume $T^k x \in H^+$ and $T^{k-1}y \not \in H^+$, hence $T^{k-1}y \in H^-$. Then $T^{k-1}y \leq T^{k+N}x$, hence $y \leq T^{N+1}x$.
\item If none of them is in $H^-$ we argue dually.
\end{itemize}
We thus obtain $x \preceq T^{N+1} y$ or $y \preceq T^{N+1}x$ in all possible cases. 
\end{proof}

\subsection{A criterion for global hyperbolicity}\label{SecHyperbolicity} We have seen in the last section how compactness of the base manifold can be used to obtain quasi-totality of a given causal covering. This sort of compactness assumption also has implications to global hyperbolicity, which we briefly want to outline here. More precisely, we will establish the folliowing:
\begin{theorem}\label{Hyperbolicity} Let $p: \widetilde{M} \to M$ be a causal covering and assume that $M$ is compact. Then the partial order $\preceq$ on $\widetilde{M}$ and its completion $\preceq_T$ are globally hyperbolic.
\end{theorem}
While up to this point we could have worked with the strict causality $\preceq_s$ instead of the closed causality $\preceq$, closedness of $\preceq$ is clearly necessary for Theorem \ref{Hyperbolicity} to hold.\\

Concerning the proof of Theorem \ref{Hyperbolicity} we first observe that the order intervals of $\preceq$ are closed by construction; since 
\[[a,b]_{\preceq_T} = \bigcup_{k=0}^{N(\widetilde{M})-1}\bigcup_{l=0}^{N(\widetilde{M})-1} [T^ka, T^{-l}b]_{\preceq},\]
we see that also $\preceq_T$ is closed. It thus remains only to show that finite order intervals of $\preceq_T$ are bounded. From now on, all order intervals will be with respect to $\preceq_T$.
Our starting point is the following trivial observation:
\begin{lemma}\label{HyperbolicityTrivial} Let $a \in \widetilde{M}$ and $N \in \mathbb N$. Then for all $x \in M$ there exists $b_0 \in p^{-1}(x)$ such that
\[p^{-1}(x) \cap [a, T^N a] \subset \{b_0, Tb_0, \dots, T^Nb_0\}.\]
\end{lemma}
\begin{proof} Let $x \in M$ and $b \in \mathcal F_x := p^{-1}(x)$. Consider
\[E_b := \{k \in \Z\,|, T^kb \succeq_T a\}.\]
We claim that $E_b$ has a minimal element. Indeed, since $T$ is dominating we have $T^la \succeq_T b$ for some $l \in \mathbb \Z$ and hence $a \succeq_T T^{-l}b$. Now if $T^{k}b \succeq_T a$, then $k \geq -l$, for otherwise $T^k b \succeq_T a \succeq_T T^{-l}b$, hence $T^{k+l} b \succeq_T b$ and thus $0 > k+l \geq 0$. Now let $k \in E_b$ be minimal and $b_0 := T^kb$. Then $T^kb_0 \geq a$ implies $k \geq 0$, while $T^k b_0 \leq T^N a$ implies $k \leq N$. Thus
\[p^{-1}(x) \cap [a, T^N a] \subset \{b_0, \dots, T^Nb_0\}.\]
\end{proof}
Now the key to the proof of Theorem \ref{Hyperbolicity} is the following lemma:
\begin{lemma} For every $x_0 \in \widetilde{M}$ there exists a bounded subset $M_{x_0} \subset \widetilde{M}$ such that $p(M_{x_0}) = M$ and $x \succeq x_0$ for all $x \in M_{x_0}$.
\end{lemma}
\begin{proof} Let $a \in M$ be some basepoint. Then there exists an open subset $U \subset M$ with the following properties:
\begin{itemize}
\item There exists a compact neighbourhood $V$ of $\overline{U}$, which is evenly covered under $p$;
\item $U$ is contractible;
\item $a \in \partial U$;
\item for every $b \in U$ there exists a causal curve $\gamma_b: [0,1] \to M$ with $\gamma_b(0) = a$, $\gamma_b(1) = b$ and $\gamma_b(t) \in U$ for all $t \in (0,1]$.
\end{itemize}
Indeed, we can take a sufficiently small open subset of $\exp({\rm Int}(C_a))$. Since $M$ is compact and weakly homogeneous there exist $g_1, \dots, g_r \in G(M, \mathcal C)$ such that
\[M = \bigcup_{j=1}^rg_jU.\]
Now choose $x_j \in p^{-1}(g_ja)$ with $x_j \succeq x_0$ and let $V_j$ be the connected component of $p^{-1}(U_j)$ containing $x_j$ in its boundary. Then we define
\[M_{x_0} := \bigcup_{j=1}^r V_j.\]
This is clearly bounded (since the closure of each $V_j$ is compact) and covers $M$. It remains to show that $x \succeq x_j$ for every $x \in V_j$; this however follows easily by lifting the curve $\gamma_{p(x)}$ to $V_j$ with basepoint $x_j$; the resulting lift is then a causal curve joining $x_j$ with $x$.
\end{proof}
Now we can easily deduce the theorem:
\begin{proof}[Proof of Theorem \ref{Hyperbolicity}] 
Let $a \in \widetilde{M}$. We use the lemma to obtain a bounded subset $M_a$ of $\widetilde{M}$ such that $p(M_{a}) = M$ and $x \succeq a$ for all $x \in M_a$. In particular we have $x \succeq_T a$ for all $x \in M_a$. We claim that there exists some $N_0 \in \mathbb N$ such that
\begin{eqnarray}\label{MAUpper}
M_a \subset [a, T^{N_0}a].
\end{eqnarray}
Indeed, suppose otherwise. Then there exists a sequence $x_n \in M_a$ with $x_n \geq T^na$. Since $M_a$ is bounded there exists a subsequence $n_k$ such that $x_{n_k} \to x \in \widetilde{M}$. Now for every $l$ there exists $k_0 \in \mathbb N$ such that for every $k \geq k_0$ we have $n_k \geq l$ and thus $x_{n_k} \succeq_T T^la$. This implies that $x_{n_k} \succeq T^{l'}a$ for some $l' \in \{l-N, \cdots, l\}$ for some fixed constant $N$. By passing to another subsequence we can thus ensure that $x_{n_{k_m}} \succeq T^{l'}a$ for $k_m \geq k_0$. Since $\succeq$ is closed this yields $x \succeq T^{l'}a$ and thus $x \succeq_T T^{l-N}a$. Since $l$ was arbitrary, this contradicts the fact that $T$ is dominating. This contradiction establishes \eqref{MAUpper}.\\

Now let $b \in \widetilde{M}$. We then find $N \in \mathbb N$ such that $N \geq N_0$ and $[a,b] \subset [a, T^{N}a]$. It remains to show that $[a, T^{N}a]$ is bounded. We claim that
\[[a, T^Na] \subset \bigcup_{n=-N}^N T^nM_a,\]
which implies the desired boundedness. Indeed, let $c \in [a, T^Na]$ and let $x := p(c)$. We consider the fiber $\mathcal F_x := p^{-1}x$. By Lemma \ref{HyperbolicityTrivial} we have
\[\mathcal F_x \cap [a, T^Na] \subset \{b_0, Tb_0, \dots, T^Nb_0\}\]
for some $b_0 \in \mathcal F_x$. In particular, we find $k_1 \in \mathbb N$ with $0 \leq k \leq N$ and $c = T^{k_1}b_0$. On the other hand we have we have \[M_a \subset [a, T^{N}a] \cap \mathcal F_x \cap [a, T^Na] \neq \emptyset\]
in view of  \eqref{MAUpper}. We thus find $k_2 \in \mathbb N$ with $0 \leq k \leq N$ and $T^{k_2}b_0 \in M_a$. Then
\[c = T^{k_1}b_0 = T^{k_1-k_2}T^{k_2}b_0 \in T^{k_1-k_2}M_a.\]
Since $-N \leq k_1-k_2 \leq N$ we obtain 
\[c \in \bigcup_{n=-N}^N T^nM_a,\]
and since $c \in [a, T^Na]$ was chosen arbitrarily, we obtain the desired boundedness result.
\end{proof}

\subsection{Examples from Lie groups}\label{SecLie} In this section we explain how the Clerc-Koufany construction of the Guichardet-Wigner quasimorphisms on a simply-connected simple Hermitian Lie groups of tube type \cite{ClercKoufany} can be reinterpreted in the language of the present paper.\\

Let $G$ be an adjoint simple Lie group with maximal compact subgroup $K$. Then $G$ is called \emph{Hermitian} if the associated symmetric space $G/K$ admits a $G$-invariant complex structure $J$, and \emph{of tube type} if $(G/K, J)$ is biholomorphic to a complex tube. From now on $G$ will always denote an adjoint simple Hermitian Lie group of tube type. We then find a Euclidean Jordan algebra $V$ such that $G/K$ can be identified with the unit ball $\mathcal D$ in $V^\C$ with respect to the spectral norm. We will fix such an identification once and for all. The action of $G$ on $\mathcal D$ extends continuously to the Shilov boundary of $\check S$. Since $\check S$ is a generalized flag manifold, we obtain a notion of transversality on $\check S$ from the associated Bruhat decomposition. It then follows from the abstract theory of generalized flag manifolds that $G$ acts almost $2$-transitively on $\check S$ (see e.g. \cite[Lemma 3.30]{Tits}). If $e_V$ denotes the unit element of the Jordan algebra $V$, then $e_V \in \check S$ and the set $\check S_{e_V}$ of points in $\check S$ transverse to $e_V$ is Zariski open in $\check S$. The Cayley transform of $V^\C$ identifies $\check S_{e_V}$ and hence $T_{-e_V}\check S$ with $V$. Thus the (closed) cone of squares in $V$ gives rise to a closed cone $\Omega \subset T_{-e_V}\check S$. By a result of Kaneyuki \cite{Kaneyuki} there exists a unique $G$-invariant causal structure $\mathcal C$ on $\check S$ with $\mathcal C_{-e_V} = \Omega$.\\ 

The universal covering $(\check R, \widetilde{\mathcal C})$ of the causal manifold $(\check S, \mathcal C)$ is described in \cite{ClercKoufany}. Namely, it turns out that $\pi_1(\check S) \cong \Z$, so that $p: \check R \to \check S$ is an infinite cyclic covering. The universal covering $\widetilde G$ of $G$ acts transitively on $\check R$; the kernel of this action can be identified with $\pi_1(G)_{tors}$. Thus $\check G := \widetilde{G}/\pi_1(G)_{tors}$ acts transitively and effectively on $\check R$. Now we claim:
\begin{proposition}
The covering $p: \check R \to \check S$ is a complete quasi-total causal covering.
\end{proposition}
\begin{proof} Identify the tangent space of $T_{-e_V}\check S$ with $V$ and choose an inner product $\langle \cdot, \cdot \rangle$ on $V$ such that $\Omega$ is a symmetric cone with respect to $\langle \cdot, \cdot \rangle$ \cite{FK}. Since $e_V$ is contained in the interior of  $\Omega$ it follows from the self-duality of the latter that
\begin{eqnarray}\label{SelfdualCone}
\exists \epsilon > 0 \, \forall x \in \Omega: \langle x, e_V \rangle \geq \epsilon \cdot \|v\|.\end{eqnarray}
Now identify $\check S$ with the compact symmetric space $K/M$, where $M$ denotes the stabilizer of $-e_V$. Since the stabilizer action of $M$ preserves both $e_V$ and the inner product, there exists a $K$-invariant $1$-form $\alpha$ on $\check S$ with
\[\alpha_{-e_V}(v) = \langle x, e_V \rangle, \quad {v \in T_{-e_V} \check S}.\]
Since $K/M$ is symmetric, this form is closed. It then follows from \eqref{SelfdualCone} that $\alpha$ is a uniformly positive $1$-form in the sense of Definition \ref{DefUnifPos}. This implies that the pullback $\beta := p^*\alpha$ is a uniformly positive $1$-form on $\check R$. In particular, $\check R$ is causal by Proposition \ref{CausalCriterion}. Since $\check S$ is a flag variety, the action of $G$ on $\check S$ is almost $2$-transitive; just take $\check S^{(2)}$ to be the set of transverse pairs in $\check S$ (see e.g. \cite{Tits}). This almost $2$-transitivity implies immediately that $\check S$ is totally acausal, whence $p: \check R \to \check S$ is a causal covering; in view of Theorem \ref{CausalQT} it also implies that this causal covering is quasi-total. It remains to show that this covering is total, i.e. $Tx \succeq x$ for all $x \in \check R$. For this it suffices to construct a causal curve joining $x$ and $Tx$; this is established in \cite{ClercKoufany}.
\end{proof}
In view of the pioneering work in \cite{Kaneyuki} we refer to the partial order $\preceq$ on $\check R$ as the \emph{Kaneyuki order}. It was established in \cite{Kaneyuki}, that $G(\check S, \check C) = G$ unless $G \cong PSL_2(\R)$. Thus let us assume $G \not \cong PSL_2(\R)$ from now on.
Then the central extension associated with the causal covering $p: \check R \to \check S$ is precisely 
\[0 \to \Z \to \check G \to G \to \{e\}.\]
We thus obtain a non-trivial quasimorphism $T_{\check R}$ on the simple Lie group $\check G$. It follows from the classification of such quasimorphisms in \cite{Shtern} that $T_{\check R}$ is necessarily a multiple of the Guichardet-Wigner quasimorphism on $\check G$ \cite{GuichardetWigner}. We have thus proved:
\begin{corollary}\label{GuiWi}
The growth functions of the order $\leq$ on $\check G$ induced from the Kaneyuki order on $\check R$ are multiples of the Guichardet-Wigner quasimorphism. 
\end{corollary}
Guichardet-Wigner quasimorphisms are very well understood; see \cite{surface} for an explicit formula. The observation that Guichardet-Wigner quasimorphisms are related to the causal structure on the corresponding Shilov boundaries was first made in \cite{ClercKoufany} (see also \cite{scl} for an English introduction to their work). However, their precise formulation of this phenomenon is different from ours. Corollary \ref{GuiWi} allows us to characterize subgroups of $\check G$ with vanishing Guichardet-Wigner quasimorphism. Indeed, as as special case of Theorem \ref{Hyperbolicity} we obtain:
\begin{corollary}\label{KaneyukiHyperbolic}
The Kaneyuki order is globally hyperbolic. 
\end{corollary}
Combining this observation with Theorem \ref{TheoremHyper} we deduce:
\begin{corollary}\label{GuiWiSubgroups}
Let $H < \check G$ be a subgroup. Then the following are equivalent:
\begin{itemize}
\item[(i)] The Guichardet-Wigner quasimorphism vanishes on $H$.
\item[(ii)] $H$ has a bounded orbit in $\check R$.
\item[(iii)] Every $H$-orbit in $\check R$ is bounded.
\end{itemize}
\end{corollary}

\appendix
\section{Partial orders on causal manifold}

\subsection{Causalities on conal manifolds}
In various branches of mathematics and physics cone fields in the tangent bundle of a manifold are used to define a \emph{causality} (i.e. a reflexive and transiive relation) on the manifold itself. The precise definitions of such causalities, however, differ widely in the literature; it thus seem worthwhile to elabarote a bit on the definitions we use in the body of text. In the present paper we are mainly interested in invariant cone fields on homogeneous spaces of finite-dimensional Lie groups, and our definitions are adapted to work well in this context. We refer the reader to   \cite{HilgertOlafsson, HHL, Lawson} for sources with a point of view similar to ours. \\

A convex, $\R^{>0}$-invariant closed subset $\Omega$ of a vector space $V$ will be called a \emph{wedge}. A wedge is called a \emph{closed cone} if it is \emph{pointed}, i.e. $\Omega \cap (-\Omega) = \{0\}$. It is called \emph{regular} if its interior is non-empty. Given a closed regular cone $\Omega \subset V$ we denote by $G(\Omega)$ the group
\[G(\Omega) := \{g \in GL(V)\,|\, g\Omega = \Omega\}.\]
Then we define:
\begin{definition} Let $M$ be a $d$-dimensional manifold and $\Omega \subset \R^d$ a closed regular cone. Then a \emph{causal structure} on $M$ is a principal $G(\Omega)$-bundle $P \to M$ together with an isomorphism $\iota: P \times_{G(\Omega)} V \to TM$ (i.e. a reduction of the structure group of $TM$ from $GL_d(\R)$ to $G(\Omega)$.) The associated fiber bundle 
\[\mathcal C := \iota(P \times_{G(\Omega)} \Omega)\]
of $P$ with fiber $\Omega$ is called the \emph{cone field} of the causal structure $P$. We then refer to the pair $(M, \mathcal C)$ as a \emph{conal manifold}. 
\end{definition}
We warn the reader that the term \emph{causal manifold} is traditionally reserved for a conal manifold with additional properties, see the definition below. We also remark that \cite{HilgertOlafsson} uses a more general definition of causal structure, but the present definition is sufficient for our purposes. For us it will be important that cone fields can be lifted along coverings:
\begin{lemma}\label{CausalCovering} Let $(M, \mathcal C)$ be a conal manifold and $\widetilde{M}$ its universal covering. Then there exists a unique cone field $\widetilde{\mathcal C}$ on
$\widetilde{M}$ such that $\pi_1(M)$ acts by causal diffeomorphisms on $(\widetilde{M}, \widetilde{\mathcal C})$. Conversely, every $\pi_1(M)$-invariant cone field descends to $M$.
\end{lemma}
\begin{proof} The only way to define a causal structure with the desired property is to set 
\[\widetilde{\mathcal C}_x = (dp_M)^{-1} \mathcal C_{p_M(x)},\]
where $dp_M: T_x\widetilde{M} \to T_{p_M(x)}M$ is the derivative of the universal covering projection. This defines indeed a causal structure on $\widetilde{M}$, since the triviality condition is local. The second statement is obvious.
\end{proof}
Given a manifold $M$ and real numbers $a < b$ we call a curve $\gamma: [a,b] \to M$ \emph{piecewise smooth} if it is continuous and there exists real numbers $a= a_0 < a_1 < \dots < a_n = b$ such that $\gamma|_{(a_j, a_{j+1})}$ is of class $C^\infty$ for $j=0, \dots, n-1$. Now we define:
\begin{definition} Let $(M, \mathcal C)$ be a conal manifold. A piecewise smooth curve $\gamma: [a,b] \to M$ is called \emph{$\mathcal C$-causal} if $\dot\gamma(t) \in \mathcal C_{\gamma(t)}$ for all but finitely many $t \in [a,b]$. The relation $\preceq_{s}$ on $M$ obtained by
setting $x \preceq_s y$ if there exists a causal curve $\gamma: [a,b] \to M$ with $\gamma(a) = x$ and $\gamma(b) = y$, is called the \emph{strict causality} of $(M, \mathcal C)$.
\end{definition}
Since the concatenation of piecewise smooth curves is piecewise smooth, the strict causality is indeed a causality. It may or may not be anti-symmetric, and it may or may not be smooth. Antisymmetry can sometimes be obtained by passing to a suitable covering. Obtaining a closed causality is difficult in general. Indeed, the closure $\preceq$ of $\preceq_s$ in $M \times M$ need no longer be transitive. Fortunately, in homogeneous examples this kind of pathology hardly occurs. To make this precise we define:
\begin{definition}\label{AutMC} Let $(M, \mathcal C)$ be a conal manifold. A diffeomorphism $\phi$ of $M$ is called \emph{causal} with respect to $\mathcal C$ if $d\phi(\mathcal C_m) = \mathcal
C_{\phi(m)}$ for all $m \in M$. The group of all causal diffeomorphisms of $(M, \mathcal C)$ is denoted $G(M, \mathcal C)$. A group action $G \times
M \to M$ is \emph{causal} if $G$ acts by causal diffeomorphisms. In this case the causal structure $\mathcal C$ is called
\emph{G-invariant}. The conal manifold $(M, \mathcal C)$ is called \emph{uniformly homogeneous} if $G(M, \mathcal C)$ acts transitively on $M$ and every $x \in M$ has an open neighbourhood $U$ such that for all $x_n \in U$ with $x_n \to x$ there exists a sequence $g_n \in G(M, \mathcal C)$ such that $g_nx_n = x$ and $g_n \to e$ in the compact-open topology.
\end{definition}
\begin{lemma}[Hilgert-Olafsson]\label{CausalClosure} Let $(M, \mathcal C)$ be a uniformly homogeneous conal manifold. Then the closure $\preceq$ of $\preceq_s$ in $M \times M$ is a causality. Moreover, the order intervals of $\preceq$ are the closures of the order intervals of $\preceq_s$.
\end{lemma}
\begin{proof} The assumption of uniform homogeneity guarantees that the proof of \cite[Prop. 2.2.4]{HilgertOlafsson} carries over.
\end{proof}
Note that if $G$ is a finite-dimensional Lie group and $H$ is a closed subgroup, then every $G$-invariant cone field on $G/H$ is uniformly homogeneous; indeed, this follows from the existence of local sections of the principal bundle $G \to G/H$. This case is actually all we need. In any case in all our examples $\preceq$ will be a well-defined causality. We follow \cite{HHL, HilgertOlafsson, Vinberg, Olshanski, Lawson} in calling $\preceq$, rather than $\preceq_s$ the causality associated with the conal manifold $(M, \mathcal C)$.

\subsection{Causal manifolds and positive $1$-forms}

In this section let $(M, \mathcal C)$ be a conal manifold. We denote by $\preceq_s$ the strict causality on $M$ and by $\preceq$ its closure. Then we define: 
\begin{definition}\label{DefCausalManifold}
The conal manifold $(M, \mathcal C)$ is called \emph{causal} if $\preceq$ is a partial order and \emph{totally acausal} if $\preceq = M \times M$. It is called \emph{weakly causal} if $\preceq_s$ is a partial order.
\end{definition}
Note that for $(M, \mathcal C)$ to be causal we demand in particular that $\preceq$ is transitive.  We now provide a sufficient condition for $M$ which guarantees causality. To this end we define:
\begin{definition}\label{DefUnifPos}
Let $(M, \mathcal C)$ be a conal manifold. A closed $1$-form $\alpha \in \Omega^1(M)$ is called \emph{uniformly positive} with respect to $\mathcal C$ if there exists a Riemannian metric on $M$ and $\epsilon > 0$ such that for all $x \in M$ and all $v \in \mathcal C_x$ we have $\alpha_x(v) \geq \epsilon \cdot \|v\|$.
\end{definition}
Uniformly positive $1$-forms are a special case of positive $1$-forms as introduced in \cite{HHL} refining ideas from \cite{Vinberg, Olshanski}. Positive $1$-forms were introduced to prove antisymmetry of the strict causality on certain $1$-connected manifolds. Uniformly positive $1$-forms play a similar role for the closed causality:
\begin{proposition}\label{CausalCriterion} Let $(M, \mathcal C)$ be a simply-connected uniformly homogeneous conal manifold admitting a uniformly positive $1$-form $\alpha$. Then $(M, \mathcal C)$ is causal.
\end{proposition}
\begin{proof} Since $M$ is uniformly homogeneous, $\preceq$ is transitive, and it remains to show that it is antisymmetric. Let $x, y \in M$ be distinct points and assume $x \preceq y \preceq x$. By definition this means that there exist sequences $x_n \to x$, $x_n' \to x$, $y_n \to y$, $y_n' \to y$ in $\check R$ such that
\[x_n \preceq_s y_n, \quad y_n' \preceq_s x_n'.\]
Let $G := G({M}, {\mathcal C})$ and observe that since $(M, \mathcal C)$ is uniformly homogeneous there exist sequence $g_n \to e$, $g_n' \to e$ in $G$ such that $y_n =g_ny$, $y_n' = g_n'y$. Now define $a_n := g_n^{-1}x_n$, $b_n := (g_n')^{-1} y_n'$. Then
\[a_n \preceq_s y \preceq_s b_n, \quad a_n \to x, \quad b_n \to x.\]
Denote by $d$ the metric induced by the Riemannian metric, for which $\alpha$ is uniformly positive. Then for any causal curve $c:[a,b] \to \widetilde{M}$ we have
\[\int_c \alpha = \int_a^b \alpha_{c(t)}(\dot c(t)) \geq L(c) \cdot \epsilon,\]
where $L(c)$ denotes the length of $c$. Fix a neighbourhood $U$ of $x$ not containing $y$ in its closure and set $\delta := \frac{\epsilon}{2}d(U, y)$. Choose $n_0$ such that $a_n, b_n \in U$ for all $n \geq n_0$. By shrinking $U$ if necessary we may assume that $U$ is geodesically convex and relatively compact. Now let $c_n$ be a causal curve from $a_n$ to $b_n$ through $y$; then $L(c_n) \geq \delta/\epsilon$ and thus 
\[\int_{c_n} \beta \geq L(c_n) \cdot \epsilon \geq \delta > 0.\]
Now denote by $c_n'$ a geodesic joining $a_n$ to $b_n$ in $U$ (parametrized by arclength) and by $(c_n')^*$ the same curve with the opposite orientation. Then the concatenation $c_n \# (c_{n}')^*$ is a closed loop in $M$; since $M$ is simply-connected this loop bounds a disc $D$ and thus
\[\int_{c_n'} \alpha = \int_{c_n} \alpha + \int_{\partial D}\alpha = \int_{c_n} \alpha + \int_{D}d\alpha = \int_{c_n} \alpha.\]
Now $\alpha$ is bounded on the compact set $\overline{U}$, hence their exists $C > 0$ such that
\[\int_{c_n'} \alpha \leq C \cdot L(c_n') = C \cdot d(a_n, b_n).\]
We have thus established for all $n > n_0$ the inequality
\[0 < \delta \leq C \cdot d(a_n, b_n).\]
Since $d(a_n, b_n) \to 0$, this is a contradiction.
\end{proof}

\bigskip

\textbf{Authors' addresses:}\\

\textsc{Departement Mathematik,  ETH Z\"urich, R\"amistr. 101, 8092 Z\"urich, Switzerland},\\
\texttt{gabi.ben.simon@math.ethz.ch};\\

\textsc{Mathematics Department, Technion - Israel Institute of Technology, Haifa, 3200, Israel},\\
\texttt{hartnick@tx.technion.ac.il}.
\end{document}